\documentclass[twoside]{amsart}

\newif\ifdraft
\draftfalse
\usepackage{amsmath,amsfonts,amsthm,mathrsfs}
\usepackage{amssymb}

\usepackage[ocgcolorlinks,unicode,bookmarks]{hyperref}
\usepackage[usenames,dvipsnames]{xcolor}
\hypersetup{colorlinks=true,citecolor=NavyBlue,linkcolor=BrickRed,urlcolor=Green} 

\usepackage{expl3}

\ExplSyntaxOn
\prop_new:N \g_cite_map_prop
\tl_new:N \l_citekey_result_tl

\cs_new:Npn \mapcitekey #1#2 {
  \clist_map_inline:nn {#2}
       {  \prop_gput:Nnn  \g_cite_map_prop  {##1} {#1}   }
}

\cs_new:Npn \getcitekey #1 {
   \prop_get:NoN \g_cite_map_prop{#1}  \l_citekey_result_tl
   \quark_if_no_value:NF \l_citekey_result_tl
       {  \tl_set_eq:NN #1  \l_citekey_result_tl  }
}

\cs_new:Npn \showcitekeymaps {\prop_show:N  \g_cite_map_prop }
\ExplSyntaxOff

\usepackage{etoolbox}
\makeatletter
\patchcmd{\@citex}{\if@filesw}{\getcitekey\@citeb \if@filesw}%
    {\typeout{*** SUCCESS ***}}{\typeout{*** FAIL ***}}
\patchcmd{\nocite}{\if@filesw}{\getcitekey\@citeb \if@filesw}%
    {\typeout{*** SUCCESS ***}}{\typeout{*** FAIL ***}}
\makeatother

\usepackage{enumitem}
\newenvironment{renumerate}{%
	\begin{enumerate}[label=(\roman{*}), ref=(\roman{*})]
}{%
	\end{enumerate}%
}

\newenvironment{aenumerate}{%
	\begin{enumerate}[label=(\alph{*}), ref=(\alph{*})]
}{%
	\end{enumerate}%
}

\usepackage{chngcntr}

\ifdraft
\usepackage[notcite,notref,color]{showkeys}

\definecolor{labelkey}{gray}{0.5}
\fi

\usepackage{tikz}
\usepackage{tikz-cd}

\newcommand{\Dmod}{\mathscr{D}}
\newcommand{\Mmod}{\mathcal{M}}

\newcommand{\shT}{\mathscr{T}}

\newcommand{\derR}{\mathbf{R}}
\newcommand{\derL}{\mathbf{L}}

\newcommand{\Ltensor}{\overset{\derL}{\tensor}}
\newcommand{\decal}[1]{\lbrack #1 \rbrack}

\newcommand{\shH}{\mathcal{H}}

\newcommand{\tensor}{\otimes}

\newcommand{\shHom}{\mathcal{H}\hspace{-1pt}\mathit{om}}

\newcommand{\ZZ}{\mathbb{Z}}
\newcommand{\QQ}{\mathbb{Q}}

\newcommand{\CC}{\mathbb{C}}

\newcommand{\PP}{\mathbb{P}}

\DeclareMathOperator{\id}{id}

\DeclareMathOperator{\Supp}{Supp}
\DeclareMathOperator{\codim}{codim}

\DeclareMathOperator{\Sym}{Sym}
\DeclareMathOperator{\gr}{gr}
\DeclareMathOperator{\DR}{DR}

\DeclareMathOperator{\Var}{Var}

\newcommand{\define}[1]{\emph{#1}}

\newcommand{\lie}[2]{\lbrack #1, #2 \rbrack}

\newcommand{\shf}[1]{\mathscr{#1}}
\newcommand{\OX}{\shf{O}_X}
\newcommand{\OmX}{\Omega_X}

\newcommand{\shV}{\shf{V}}

\def\overbar#1#2#3{{%
	\setbox0=\hbox{\displaystyle{#1}}%
	\dimen0=\wd0
	\advance\dimen0 by -#2 
	\vbox {\nointerlineskip \moveright #3 \vbox{\hrule height 0.3pt width \dimen0}%
		\nointerlineskip \vskip 1.5pt \box0}%
}}

\newcommand{\into}{\hookrightarrow}

\newcommand{\jl}{j_{\ast}}

\newcommand{\fu}{f^{\ast}}
\newcommand{\fl}{f_{\ast}}

\newcommand{\pu}{p^{\ast}}
\newcommand{\pl}{p_{\ast}}

\newcommand{\tl}{t_{\ast}}

\newcommand{\hu}{h^{\ast}}
\newcommand{\varphiu}{\varphi^{\ast}}
\newcommand{\varphil}{\varphi_{\ast}}
\newcommand{\hl}{h_{\ast}}

\newcommand{\hp}{h_{+}}

\newcommand{\shF}{\shf{F}}
\newcommand{\shG}{\shf{G}}
\newcommand{\shE}{\shf{E}}

\newcommand{\shO}{\shf{O}}

\makeatletter
\let\@@seccntformat\@seccntformat
\renewcommand*{\@seccntformat}[1]{%
  \expandafter\ifx\csname @seccntformat@#1\endcsname\relax
    \expandafter\@@seccntformat
  \else
    \expandafter
      \csname @seccntformat@#1\expandafter\endcsname
  \fi
    {#1}%
}
\newcommand*{\@seccntformat@subsection}[1]{%
  \textbf{\csname the#1\endcsname.}
}
\makeatother

\makeatletter
\let\@paragraph\paragraph
\renewcommand*{\paragraph}[1]{%
	\vspace{0.3\baselineskip}%
	\@paragraph{\textit{#1}}%
}
\makeatother

\counterwithin{equation}{subsection}
\counterwithout{subsection}{section}
\counterwithin{figure}{subsection}

\newtheorem{theorem}[equation]{Theorem}
\newtheorem*{theorem*}{Theorem}
\newtheorem{lemma}[equation]{Lemma}
\newtheorem*{lemma*}{Lemma}
\newtheorem{corollary}[equation]{Corollary}
\newtheorem*{corollary*}{Corollary}
\newtheorem{proposition}[equation]{Proposition}
\newtheorem*{proposition*}{Proposition}
\newtheorem{conjecture}[equation]{Conjecture}
\newtheorem*{conjecture*}{Conjecture}

\theoremstyle{definition}
\newtheorem{definition}[equation]{Definition}
\newtheorem*{definition*}{Definition}
\theoremstyle{remark}

\newtheorem{example}[equation]{Example}
\newtheorem*{example*}{Example}
\newtheorem*{problem*}{Problem}
\newtheorem*{note}{Note}

\theoremstyle{plain}

\newcommand{\theoremref}[1]{\hyperref[#1]{Theorem~\ref*{#1}}}
\newcommand{\lemmaref}[1]{\hyperref[#1]{Lemma~\ref*{#1}}}
\newcommand{\definitionref}[1]{\hyperref[#1]{Definition~\ref*{#1}}}
\newcommand{\propositionref}[1]{\hyperref[#1]{Proposition~\ref*{#1}}}
\newcommand{\conjectureref}[1]{\hyperref[#1]{Conjecture~\ref*{#1}}}
\newcommand{\corollaryref}[1]{\hyperref[#1]{Corollary~\ref*{#1}}}
\newcommand{\exampleref}[1]{\hyperref[#1]{Example~\ref*{#1}}}
\newcommand{\exerciseref}[1]{\hyperref[#1]{Exercise~\ref*{#1}}}

\makeatletter
\let\old@caption\caption
\renewcommand*{\caption}[1]{%
	\setcounter{figure}{\value{equation}}%
	\stepcounter{equation}%
	\old@caption{#1}\relax%
}
\makeatother

\newcounter{intro}

\newtheorem{intro-conjecture}[intro]{Conjecture}
\newtheorem{intro-corollary}[intro]{Corollary}
\newtheorem{intro-theorem}[intro]{Theorem}

\newcommand{\omY}{\omega_Y}

\newcommand{\OmY}{\Omega_Y}
\newcommand{\OY}{\shO_Y}

\newcommand{\df}{\mathit{df}}

\newcommand{\parref}[1]{\hyperref[#1]{\S\ref*{#1}}}
\newcommand{\chapref}[1]{\hyperref[#1]{Chapter~\ref*{#1}}}

\makeatletter
\newcommand*\if@single[3]{%
  \setbox0\hbox{${\mathaccent"0362{#1}}^H$}%
  \setbox2\hbox{${\mathaccent"0362{\kern0pt#1}}^H$}%
  \ifdim\ht0=\ht2 #3\else #2\fi
  }
\newcommand*\rel@kern[1]{\kern#1\dimexpr\macc@kerna}
\newcommand*\widebar[1]{\@ifnextchar^{{\wide@bar{#1}{0}}}{\wide@bar{#1}{1}}}
\newcommand*\wide@bar[2]{\if@single{#1}{\wide@bar@{#1}{#2}{1}}{\wide@bar@{#1}{#2}{2}}}
\newcommand*\wide@bar@[3]{%
  \begingroup
  \def\mathaccent##1##2{%
    \if#32 \let\macc@nucleus\first@char \fi
    \setbox\z@\hbox{$\macc@style{\macc@nucleus}_{}$}%
    \setbox\tw@\hbox{$\macc@style{\macc@nucleus}{}_{}$}%
    \dimen@\wd\tw@
    \advance\dimen@-\wd\z@
    \divide\dimen@ 3
    \@tempdima\wd\tw@
    \advance\@tempdima-\scriptspace
    \divide\@tempdima 10
    \advance\dimen@-\@tempdima
    \ifdim\dimen@>\z@ \dimen@0pt\fi
    \rel@kern{0.6}\kern-\dimen@
    \if#31
      \overline{\rel@kern{-0.6}\kern\dimen@\macc@nucleus\rel@kern{0.4}\kern\dimen@}%
      \advance\dimen@0.4\dimexpr\macc@kerna
      \let\final@kern#2%
      \ifdim\dimen@<\z@ \let\final@kern1\fi
      \if\final@kern1 \kern-\dimen@\fi
    \else
      \overline{\rel@kern{-0.6}\kern\dimen@#1}%
    \fi
  }%
  \macc@depth\@ne
  \let\math@bgroup\@empty \let\math@egroup\macc@set@skewchar
  \mathsurround\z@ \frozen@everymath{\mathgroup\macc@group\relax}%
  \macc@set@skewchar\relax
  \let\mathaccentV\macc@nested@a
  \if#31
    \macc@nested@a\relax111{#1}%
  \else
    \def\gobble@till@marker##1\endmarker{}%
    \futurelet\first@char\gobble@till@marker#1\endmarker
    \ifcat\noexpand\first@char A\else
      \def\first@char{}%
    \fi
    \macc@nested@a\relax111{\first@char}%
  \fi
  \endgroup
}
\makeatother

\newcommand{\omX}{\omega_X}
\newcommand{\omYX}{\omega_{Y/X}}
\newcommand{\omZX}{\omega_{Z/X}}
\newcommand{\omZ}{\omega_Z}
\newcommand{\OmZ}{\Omega_Z}
\newcommand{\OZ}{\shO_Z}
\newcommand{\shQ}{\mathscr{Q}}
\newcommand{\shA}{\mathcal{A}}

\newcommand{\shEb}{\shE_{\bullet}}
\newcommand{\shFb}{\shF_{\bullet}}
\newcommand{\shGb}{\shG_{\bullet}}
\DeclareMathOperator{\HM}{HM}
\newcommand{\grFMb}{\gr_{\bullet}^F \! \Mmod}
\newcommand{\grFM}{\shG^M}
\renewcommand{\shV}{\mathcal{V}}
\newcommand{\shVt}{\tilde{\shV}}
\newcommand{\shVg}{\shVt^{\geq 0}}

\mapcitekey{Campana+Peternell:GeometricStability}{CP}
\mapcitekey{Campana+Paun:Quotients}{CPa}
\mapcitekey{Kollar:Subadditivity}{Kollar}
\mapcitekey{Viehweg:WP1}{Viehweg1}
\mapcitekey{Viehweg:WP2}{Viehweg2}
\mapcitekey{Viehweg:Survey}{Viehweg3}
\mapcitekey{Viehweg:Hilbert1}{Viehweg4}
\mapcitekey{Popa+Wu:WeakPositivity}{PW}
\mapcitekey{Viehweg+Zuo:Isotriviality}{VZ1}
\mapcitekey{Viehweg+Zuo:BaseSpaces}{VZ2}
\mapcitekey{Viehweg+Zuo:Brody}{VZ3}
\mapcitekey{Zuo:Kernels}{Zuo}
\mapcitekey{Popa+Schnell:OneForms}{PS}
\mapcitekey{Popa+Schnell:mhmgv}{mhmgv}
\mapcitekey{Saito:KC}{Saito-KC}
\mapcitekey{BDPP:Pseudoeffective}{BDPP}
\mapcitekey{Esnault+Viehweg:VanishingTheorems}{EV}
\mapcitekey{Saito:HodgeModules}{Saito-MHP}
\mapcitekey{Saito:MixedHodgeModules}{Saito-MHM}
\mapcitekey{Popa:SaitoVanishing}{Popa}
\mapcitekey{Brunebarbe:Variations}{Brunebarbe}
\mapcitekey{Schnell:HolonomicAbelian}{Schnell}
\mapcitekey{Kebekus+Kovacs:FamiliesSurfaces}{KK1}
\mapcitekey{Kebekus+Kovacs:FamiliesCompact}{KK2}
\mapcitekey{Kebekus+Kovacs:FamiliesThreefolds}{KK3}
\mapcitekey{Kovacs:FamiliesSurvey}{Kovacs}
\mapcitekey{Kebekus:VZSurvey}{Kebekus}
\mapcitekey{Patakfalvi:Hyperbolicity}{Patakfalvi}
\mapcitekey{Kawamata:AdditivityMMP}{Kawamata}
\mapcitekey{Mori:Survey}{Mori}
\mapcitekey{Schnell:sanya}{sanya}
\mapcitekey{Taji:special}{Taji}
\mapcitekey{Deligne:EquationsDifferentielles}{Deligne}
\mapcitekey{Pareschi+Popa:cdf}{PP}
\mapcitekey{Kashiwara:VanishingCycle}{Kashiwara}

\begin{document}

\title[Viehweg's hyperbolicity conjecture]{Viehweg's hyperbolicity conjecture for families with maximal variation}

\author{Mihnea Popa}
\address{Department of Mathematics, Northwestern University,
2033 Sheridan Road, Evanston, IL 60208, USA} 
\email{mpopa@math.northwestern.edu}

\author{Christian Schnell}
\address{Department of Mathematics, Stony Brook University, Stony Brook, NY 11794-3651}
\email{cschnell@math.sunysb.edu}


\begin{abstract}
We use the theory of Hodge modules to construct Viehweg-Zuo sheaves on base spaces of families with maximal 
variation and fibers of  general type; and, more generally, families whose geometric generic fiber has a good minimal model. Combining 
this with a result of Campana-P\u aun, we 
deduce Viehweg's hyperbolicity conjecture in this context, namely the fact that the base spaces of such families 
are of log general type. This is approached as part of a general problem of identifying what spaces can support Hodge theoretic objects with certain positivity properties.
\end{abstract}

\date{\today}

\setlength{\parskip}{0.05\baselineskip}

\maketitle



\section{Introduction}

\subsection{Families of varieties}
The main aim of this paper is to give a proof of Viehweg's hyperbolicity conjecture
for base spaces of families of varieties of general type with maximal variation, and
more generally, when assuming the conjectures
of the minimal model program, for arbitrary families with maximal variation.

\begin{intro-theorem}\label{thm:hyperbolicity}
Let $f: Y \rightarrow X$ be an algebraic fiber space between smooth projective
varieties, and let $D \subseteq X$ be any divisor containing the singular locus of $f$. Assume that
$f$ has maximal variation, in the sense that $\Var (f) = \dim X$. Then:
\begin{renumerate}
\item If the general fiber of $f$ is of general type, then the pair $(X, D)$ is of
log-general type, meaning that $\omX (D)$ is big. 
\item More generally, the same conclusion holds if the geometric generic fiber of $f$ admits a good minimal model.
\end{renumerate}
\end{intro-theorem}

We obtain \theoremref{thm:hyperbolicity} as a consequence of the main result we prove, regarding the
existence of what are sometimes called Viehweg-Zuo sheaves, stated below, combined
with a key theorem of Campana-P\u{a}un \cite{CPa} on the pseudo-effectivity of quotients of powers of log-cotangent bundles.

\begin{intro-theorem}\label{thm:VZ-sheaves}
Let $f: Y \rightarrow X$ be an algebraic fiber space between smooth projective
varieties, such that the $f$-singular locus $D_f \subseteq X$ is a simple normal crossing divisor.
Assume that for every generically finite $\tau: \tilde X \rightarrow X$ with $\tilde
X$ smooth, and for every resolution $\tilde Y$ of $Y\times_X \tilde X$, there is an
integer $m \geq 1$ such that 
$\det {\tilde f}_* \omega_{\tilde Y/ \tilde X}^{\otimes m}$ is big. Then there exists a big coherent sheaf $\shH$ on $X$ and an integer 
$s \ge 1$, together with an inclusion
$$\shH \hookrightarrow \big(\OmX^1 (\log D_f) \big)^{\otimes s}.$$
\end{intro-theorem}

Going back to the statement of \theoremref{thm:hyperbolicity}, the variation $\Var
(f)$ is an invariant introduced by Viehweg \cite{Viehweg1} in order to measure how
much the birational isomorphism class of the fibers of $f$ varies along $Y$. Maximal
variation $\Var(f) = \dim X$ simply means that the general fiber can only be
birational to countably many other fibers. The connection between the two theorems is made
via Viehweg's $Q_{n,m}$ conjecture, which states that if $f$ has maximal variation,
then $\det \fl \omYX^{\otimes m}$ is big for some $m \geq 1$.\footnote{By the
determinant of a torsion-free sheaf $\shF$ of generic rank $r$, we mean $(\bigwedge^r
\shF)^{\vee \vee}$.} This implies Viehweg's $C_{n,m}^+$ generalization of Iitaka's
conjecture, see \cite[Theorem II]{Viehweg1} and \cite[Remark 3.7]{Viehweg2}, and was
shown to hold when the fibers are of general type by Koll\'ar \cite{Kollar} (see also
\cite{Viehweg4}). Moreover, Kawamata \cite[Theorem 1.1]{Kawamata} proved that
that $Q_{n,m}$ holds for any morphism whose geometric generic fiber has a good
minimal model.

\subsection{Previous results}
Viehweg's original conjecture (see \cite[6.3]{Viehweg3} or \cite[Problem 1.5]{VZ2})
is a generalization of Shafarevich's conjecture on non-isotrivial one-parameter
families of curves:  it states that if $X^{\circ}$ is smooth and quasi-projective,
and $f^{\circ}: Y^{\circ} \rightarrow X^{\circ}$ is a smooth family of canonically
polarized varieties with maximal variation,  then $X^\circ$ must be of log-general
type.\footnote{It is standard that $f^{\circ}$ can be compactified to a morphism $f:
X\rightarrow Y$ whose singular locus $D = X \setminus X^{\circ}$  is a divisor. In the paper we phrase things directly in this set-up, but note that the conclusions should be seen as properties of the original family $f^{\circ}$.} This is of course very much related to the study of subvarieties of moduli stacks.

In this setting, i.e. for families of canonically polarized varieties, 
\theoremref{thm:hyperbolicity} is by now fully known. This is due
to important work of many authors, all relying crucially on the existence of
Viehweg-Zuo sheaves \cite[Theorem 1.4]{VZ2} for such families;  we briefly
review the main highlights, without providing an exhaustive list (but see also
\cite{Kovacs} and \cite{Kebekus} for comprehensive surveys and references to previous
work over one-dimensional bases). The result was shown by Kebekus-Kov\'acs when $X$
is a surface in \cite{KK1}, and then in \cite{KK2} when $D= \emptyset$ assuming the
main conjectures of the minimal model program, while \cite{KK3} contains more refined results in
dimension at most three. It was then deduced unconditionally by Patakfalvi
\cite{Patakfalvi} from the results of \cite{CP}, when $D = \emptyset$ and when $X$ is
not uniruled. Finally, Campana-P\u aun obtained the result in general,
based on their bigness criterion \cite[Theorem 7.7]{CPa} that we use here as well. Using further 
results from \cite{VZ2}, it is possible to extend this argument to families of varieties with semiample
canonical bundle.

The work of Viehweg-Zuo \cite{VZ1}, \cite{VZ2} suggested however that Viehweg's conjecture should hold much more generally, 
indeed for all families with maximal variation.  Note that the full \theoremref{thm:hyperbolicity} in the case $\dim X = 1$  was proved in \cite[Theorem 0.1]{VZ1}.  Our work owes a lot to the general strategy introduced in their papers, as we will see below.

\subsection{Kebekus-Kov\'acs and Campana conjectures}
In \cite{KK1} Kebekus and Kov\'acs proposed a natural extension of Viehweg's conjecture taking into account 
families of canonically polarized varieties that are not necessarily of maximal variation. At least when $\kappa (X^\circ) \ge 0$ it predicts that one should have 
$\kappa (X^\circ) \ge {\rm Var}(f)$;  see \conjectureref{generalized_conjecture} for the precise statement and the results they obtained.
We remark in \S\ref{scn:KK} that, at least when $X^\circ$ is projective, the methods of this paper also apply to the extension of this  conjecture to families whose geometric generic fiber has a good minimal model, assuming a positive answer to an abundance-type conjecture of Campana-Peternell.

A closely related problem is Campana's isotriviality conjecture, which predicts that
smooth families of canonically polarized varieties over a special base must be
isotrivial, and which  implies the Kebekus-Kov\'acs conjecture. This conjecture has recently been proved by Taji \cite{Taji}, 
It would be interesting to know whether our construction (based on Hodge modules) can be adapted
to the case of orbifolds.

\subsection{Abstract results on Hodge modules and Higgs bundles}
The main new ingredient for proving \theoremref{thm:VZ-sheaves} is the construction in \chapref{chap:Hodge}
of certain Hodge modules, and of Higgs bundles derived from them. They are associated to morphisms whose relative canonical bundle satisfies a mild positivity condition; see ($\ref{eq:sections}$) below. 
Such Hodge modules and Higgs bundles satisfy a ``largeness" property for the first
step in the Hodge filtration, defined in \S\ref{big-HM}. The existence of a
Viehweg-Zuo sheaf, meaning a big sheaf as in \theoremref{thm:VZ-sheaves}, turns out to be an instance of an abstract result about Hodge modules with this property.
Oversimplifying in order to explain the main idea, we consider pure Hodge modules 
with the property that there exist a big line bundle $A$ and a sheaf inclusion
\begin{equation}\label{eqn:first}
A \hookrightarrow F_{p(\Mmod)} \Mmod \otimes \shO_X(\ell D)
\end{equation}
where $F_{p(\Mmod)} \Mmod$ is the lowest non-zero term in the  filtration
$F_{\bullet} \Mmod$ on the underlying $\Dmod$-module $\Mmod$, $D$ is a divisor away from which $M$ is a variation of Hodge structure, and $\ell \ge 0$ is an integer.  This can be seen as an abstract version of the more familiar property of the period map being immersive at a point.

\begin{intro-theorem}\label{thm:HM-positivity}
Let $X$ be a smooth projective variety, and $M$ a polarizable pure Hodge module with strict support $X$, extending a variation of Hodge structure of weight $k$ on a dense open subset $U = X \setminus D$, with $D$ a divisor. If its underlying filtered  
$\Dmod_X$-module $(\Mmod, F_{\bullet}\Mmod)$ satisfies ($\ref{eqn:first}$), then at least one of the following holds:

\begin{renumerate}
\item $D$ is big.
\item There exists $1 \le s \le k$, $r \ge 1$, and a big coherent sheaf $\shG$ on $X$
such that  $$\shH \hookrightarrow (\OmX^1)^{\otimes s} \otimes \shO(rD).$$
\end{renumerate}
Consequently, if $X$ is not uniruled, then $\omX (D)$ is big.
\end{intro-theorem}

Therefore variations of Hodge structure can have such large extensions only if they
are supported on the complement of a sufficiently positive divisor.\footnote{Another example of this phenomenon is  \cite[Theorem 26.2]{Schnell}, which says that $D$ must be ample when $X$ is an abelian variety.} When the base is not uniruled this implies by the result of  \cite{CPa} that it must actually be of log-general type. In this abstract context the non-uniruledness hypothesis is necessary; see \exampleref{ex:uniruled-counterex}.

In reality, for applications like \theoremref{thm:VZ-sheaves} a more refined setup is needed. Besides the polarizable Hodge module $M$,  we also need to consider a graded $\Sym \shT_X$-submodule  $\shGb \subseteq \gr_{\bullet}^F \! \Mmod$ of the associated graded of the underlying filtered $\Dmod_X$-module $(\Mmod, F_{\bullet} \Mmod)$. This is constructed in \theoremref{thm:VZ}.
In the case when $X$ is an abelian variety, this type of construction was considered in \cite{PS}. 
We prove and use a slightly stronger version of \theoremref{thm:HM-positivity} involving such submodules; see 
 \theoremref{thm:submodule-positivity} for the precise statement, which uses a recent weak positivity result for Hodge modules from \cite{PW}.

To deal with the case when $X$ is uniruled, an additional step is needed. Using the
pair $(M, \shGb)$ from above, we produce 
in \theoremref{thm:Higgs} another pair $(\shEb, \shFb)$, where $(\shEb, \theta)$ is a
(graded logarithmic) Higgs bundle with poles along a divisor 
$D$ containing $D_f$, while $\shFb \subseteq \shEb$ is a subsheaf such that
\[
	\theta(\shFb) \subseteq  \Omega^1_X (\log D_f) \otimes \shF_{\bullet+1},
\]
and which again satisfies a largeness property. This uses some of the more technical
aspects of the theory of Hodge modules, especially the interaction between the Hodge filtration and the $V$-filtration along hypersurfaces. The corresponding version of \theoremref{thm:HM-positivity} is  \theoremref{thm:submodule-Higgs}, which finally produces the Viehweg-Zuo sheaf best suited for our purposes.

\subsection{Outline of the proof}

We summarize the discussion above into a brief outline of the proof of \theoremref{thm:VZ-sheaves} and \theoremref{thm:hyperbolicity}.

\begin{enumerate}
\item Due to results of Kawamata, Koll\'ar, and Viehweg,  families with maximal
variation whose geometric generic fiber has a good minimal model satisfy the
hypothesis of \theoremref{thm:VZ-sheaves}.
\item Given a big line bundle $L$ on $X$, we show that the
$m$-th power of the line bundle $\omYX \tensor \fu L^{-1}$ has a nontrivial global
section for some $m \ge 1$. This uses the bigness of $\det \fl \omega_{\tilde Y/ \tilde X}^{\tensor m}$ 
on generically finite covers of $X$, Viehweg's fiber product trick, and semistable reduction; in the process, we replace $Y$
by a resolution of singularities of a very large fibered product $Y \times_X \dotsm
\times_X Y$.
\item We construct a polarizable Hodge module $M$ on $X$, together with a graded
$\Sym \shT_X$-submodule $\shGb \subseteq \grFMb$, such that $\shG_0 = L$ and $\Supp
\shG \subseteq S_f$, the set of singular cotangent vectors of $f$ in $T^* X$. 
Both are obtained from a resolution of singularities of a
branched covering determined by the section in (2). After resolving
singularities, we may assume that $M$ restricts to a variation of Hodge structure
outside a normal crossing divisor $D \supseteq D_f$.
\item The variation of Hodge structure on $X \setminus D$ determines a (graded
logarithmic) Higgs bundle $\shEb$ with Higgs field
\[
	\theta \colon \shEb \to \OmX^1(\log D) \tensor \shE_{\bullet+1}.
\]
From $\shGb$ in (3), we construct 
a subsheaf $\shFb \subseteq \shEb$ such that
$\shF_0$ is a big line bundle and $\theta(\shFb) \subseteq \OmX^1(\log D_f) \tensor
\shF_{\bullet+1}$. This uses the interaction between the Hodge filtration and the $V$-filtration on Hodge modules.
\item We deduce that some large tensor power of $\OmX^1(\log D_f)$ contains a big subsheaf  
(or Viehweg-Zuo sheaf), concluding the proof of \theoremref{thm:VZ-sheaves}.
The main ingredient is a slight extension of a theorem of Zuo, 
to the effect that the dual of the kernel of $\theta \colon \shEb \to \OmX^1(\log D) \tensor
\shE_{\bullet+1}$ is weakly positive.
\item Finally, to deduce \theoremref{thm:hyperbolicity}, we can assume after a birational modification that the $f$-singular locus 
$D_f$ is a divisor with simple normal crossings. We apply \theoremref{thm:VZ-sheaves} and a recent theorem by Campana and P\u{a}un to conclude that, 
in this situation, the line bundle $\det \OmX^1(\log D_f) = \omega_X ( D_f)$ is big. This
proves that the pair $(X, D_f)$ is of log general type.
\end{enumerate}

These steps are addressed throughout the paper, and are collected together in
\S\ref{scn:general}. We also include in \S\ref{scn:non-uniruled} a substantially
simpler proof of \theoremref{thm:hyperbolicity} in the case when $X$ is not uniruled.
It avoids many of the technicalities involved in dealing with the remaining case,
while still containing all the key ideas in a particularly transparent form; hence it
may help at a first reading. In brief, it only needs a less precise version of (1),
which does not use semistable reduction; it does not need (4), which is the most
technical part of the general proof, and it replaces (5) with similar results applied directly to the Hodge module construction in (2).

\subsection{What is new}
As mentioned above, our work owes to the beautiful approach of Viehweg and Zuo
\cite{VZ1, VZ2, VZ3} to the study of families with maximal variation, by means of constructing Viehweg-Zuo sheaves as a main step towards understanding the base space of such a family. For families of varieties with semiample canonical bundle they constructed 
Higgs systems as in (4) above, using a very delicate analysis based on weak positivity, while dealing with mild enough singularities due to the semiampleness assumption. They pioneered the idea of using negativity results for Kodaira-Spencer 
kernels to extract positivity from the Higgs systems thus constructed.

This paper offers two main new inputs. The first is to view the hyperbolicity problem as a special case of the study of spaces supporting abstract Hodge theoretic objects satisfying the largeness property described above, an interesting problem in its own right.
The second, and most significant, is the use of Hodge modules. Just as in \cite{PS}, we are able to address a more general situation due to the fact that Hodge modules provide higher flexibility in dealing with the singularities created by section produced in (2), and for applying positivity results. Even in the case of canonically polarized fibers this simplifies the argument in \cite{VZ2}, at least when the base is not uniruled. In general however, besides appealing to a Hodge module construction, providing a Higgs system with all the necessary properties requires the use of some quite deep input from Saito's theory.

\subsection{Acknowledgement}

The first author is grateful for the hospitality of the Department of Mathematics at
the University of Michigan, where he completed part of this work.  Both authors thank Fr\'ed\'eric Campana, 
Stefan Kebekus, S\'andor Kov\'acs, Mihai P\u{a}un, and Behrouz Taji for useful discussions. They also 
thank the referee for suggestions that improved the exposition.
During the preparation of the paper, M. Popa was partially supported by NSF grant
DMS-1405516 and a Simons Fellowship, and Ch. Schnell by NSF grant DMS-1404947 and a
Centennial Fellowship of the American Mathematical Society.

\section{Construction of Hodge modules and Higgs bundles}
\label{chap:Hodge}

Let $f \colon Y \to X$ be a surjective morphism between two smooth projective
varieties. In this chapter, we describe a general method for obtaining information
about the $f$-singular locus $D_f \subseteq X$ of the morphism from positivity
assumptions on the relative canonical bundle $\omYX$. 

\subsection{Cotangent bundles}

We begin by introducing a more refined measure, inside the cotangent bundle, for the
singularities of a morphism. Given a surjective morphism $f \colon Y \to X$ between
two smooth projective varieties, we use the following notation for the induced
morphisms between the cotangent bundles of $X$ and $Y$.
\begin{equation} \label{eq:Japanese}
\begin{tikzcd}
Y \dar{f} & T^{\ast} X \times_X Y \lar[swap]{p_2} \dar{p_1} \rar{\df} & T^{\ast} Y \\
X & T^{\ast} X \lar[swap]{p}
\end{tikzcd}
\end{equation}
Inside the cotangent bundle of $X$, consider the set of \emph{singular cotangent
vectors}
\[
	S_f = p_1 \bigl( \df^{-1}(0) \bigr) \subseteq T^{\ast} X.
\]
A cotangent vector $(x, \xi) \in T^{\ast} X$ belongs to $S_f$ if and
only if $\fu \xi$ vanishes at some point $y \in f^{-1}(x)$, or equivalently, if $\xi$
annihilates the image of $T_y Y$ inside the tangent space $T_x X$. If this happens
and $\xi \neq 0$, then $f$ is not submersive at $y$, which means that $x$ belongs to
the $f$-singular locus $D_f \subseteq X$. Consequently, $S_f$ is the union of the zero
section and a closed conical subset of $T^{\ast} X$ whose image under the projection
$p \colon T^{\ast} X \to X$ is equal to $D_f$.
The set of singular cotangent vectors $S_f$ has the following interesting property.  

\begin{lemma} \label{lem:Sf}
One has $\dim S_f \leq \dim X$, and every irreducible component of $S_f$ of dimension
$\dim X$ is the conormal variety of a subvariety of $X$.
\end{lemma}

\begin{proof}
Fix an irreducible component $W \subseteq S_f$, and denote by $Z = p(W)$ its image in
$X$; because $S_f$ is conical, this is a closed subvariety of $X$. Both
assertions will follow if we manage to show that $W$ is contained in the conormal
variety
\[
	T_Z^{\ast} X 
		= \text{closure in $T^{\ast} X$ of the conormal bundle to the smooth locus of $Z$.}
\]
By definition, for every cotangent vector $(x, \xi) \in W$, there is a point $y \in
f^{-1}(x)$ such that $\xi$ vanishes on the image of $T_y Y \to T_x X$. At a general
smooth point $x \in Z$, this image contains the tangent space $T_x Z$, and so
$(x, \xi) \in T_Z^{\ast} X$. 
\end{proof}

The zero section is clearly one of the irreducible components of $S_f$; for dimension
reasons, the conormal varieties of the divisorial components of $D_f$ are also
contained in $S_f$. Other irreducible components of dimension $\dim X$ are less easy to
come by. 

Evidently, the morphism $f$ is smooth if and only if $S_f$ is equal to the zero
section. One method for getting a lower bound on the size of $S_f$ -- and, therefore,
of $D_f$ -- is to look for coherent sheaves on $T^{\ast} X$ whose support is contained in
the set $S_f$. In practice, it is better to work with sheaves of graded modules over
the symmetric algebra $\shA_X = \Sym \shT_X$, where $\shT_X$ is the tangent sheaf of
$X$. Recall that 
\[
	\shA_X \simeq \pl \shO_{T^{\ast} X},
\]
and that taking the direct image under $p \colon T^{\ast} X \to X$ gives an
equivalence of categories between (algebraic) coherent sheaves on the cotangent bundle
and coherent $\shA_X$-modules. For a coherent graded $\shA_X$-module 
\[
	\shGb = \bigoplus_{k \in \ZZ} \shG_k,
\]
we use the symbol $\shG$, without the dot, to denote the associated coherent sheaf on
$T^{\ast} X$; it has the property that $\pl \shG \simeq \shGb$ as modules over
$\shA_X$ (without the grading).

\subsection{The main result}

For the remainder of the chapter, let us fix a surjective morphism $f \colon Y \to X$
between two smooth projective varieties. We also fix a line bundle $L$ on $X$, and
consider on $Y$ the line bundle
\[
	B = \omYX \tensor \fu L^{-1}.
\]
We assume that the following condition holds: 
\begin{equation}\label{eq:sections}
	H^0 (Y, B^{\otimes m}) \neq 0 \quad \text{for some $m \geq 1$.}
\end{equation}
Starting from this data, we construct a graded module over $\Sym \shT_X$ with the
following properties.

\begin{theorem} \label{thm:VZ}
Assuming \eqref{eq:sections}, one can find a graded $\shA_X$-module $\shGb$
that is coherent over $\shA_X$ and has the following properties:
\begin{aenumerate}
\item \label{en:VZa}
As a coherent sheaf on the cotangent bundle, $\Supp \shG \subseteq S_f$.
\item \label{en:VZb}
One has $\shG_0 \simeq L \tensor \fl \OY$.
\item \label{en:VZc}
Each $\shG_k$ is torsion-free on the open subset $X \setminus D_f$. 
\item \label{en:VZd}
There exists a regular holonomic $\Dmod$-module $\Mmod$ with good filtration
$F_{\bullet} \Mmod$, and an inclusion of graded $\shA_X$-modules $\shGb \subseteq
\gr_{\bullet}^F \! \Mmod$.
\item \label{en:VZe}
The filtered $\Dmod$-module $(\Mmod, F_{\bullet} \Mmod)$ underlies a polarizable Hodge module $M$ on
$X$ with strict support $X$, and $F_k \Mmod = 0$ for $k < 0$.
\end{aenumerate}
\end{theorem}

Roughly speaking, $\shGb$ is constructed by applying results from the theory of Hodge
modules to a resolution of singularities of a branched covering of $Y$. This type of
construction was invented by Viehweg and Zuo, but it becomes both simpler and more
flexible through the use of Hodge modules. An introduction to Hodge modules that
covers all the results we need here, with references to the original work of Saito, can be found in \cite{sanya}.

Despite the technical advantages of working with Hodge modules, the proof of
Viehweg's conjecture in the general case works more naturally in the context of 
Higgs sheaves (with logarithmic poles along $D_f$).  Since we do not have any control
over the zero locus of the section
in \eqref{eq:sections}, we do not attempt to construct such an object directly from
the branched covering. Instead, we use the local properties of Hodge modules to
construct a suitable Higgs sheaf from the graded module $\shGb$, at least on
some birational model of $X$. More precisely, we can perform finitely many blowups
with smooth centers to assume that the $f$-singular locus $D_f$ is a divisor, and moreover to 
put ourselves in the following situation:
\begin{equation} \label{eq:NCD}
	\text{The singularities of $M$ occur along a normal crossing 
	divisor $D \supseteq D_f$.}
\end{equation}
Concretely, this means that the restriction of $M$ to the open subset $X \setminus D$
is a polarizable variation of Hodge structure. We use this fact to construct from
$\shGb$ an $\OX$-module $\shFb$ with the structure of a graded Higgs sheaf.

\begin{theorem} \label{thm:Higgs}
Let $f \colon Y \to X$ be a surjective morphism with connected fibers between two
smooth projective varieties. Assuming \eqref{eq:sections} and \eqref{eq:NCD}, one
can find an $\OX$-module $\shFb$ with the following properties:
\begin{aenumerate}
\item \label{en:Higgsb}
One has $L(-D_f) \subseteq \shF_0 \subseteq L$.
\item \label{en:Higgsc}
Each $\shF_k$ is a reflexive coherent sheaf on $X$.
\item \label{en:Higgsd}
There exists a (graded logarithmic) Higgs bundle $\shEb$ with Higgs field 
\[
	\theta \colon \shEb \to \OmX^1(\log D) \tensor \shE_{\bullet+1}, 
\]
such that $\shFb \subseteq \shEb$ and $\theta(\shFb) \subseteq \OmX^1(\log D_f)
\tensor \shF_{\bullet+1}$.
\item \label{en:Higgse}
The pair $(\shEb, \theta)$ comes from a polarizable variation of Hodge structure on
$X \setminus D$ with $\shE_k = 0$ for $k < 0$.
\end{aenumerate}
\end{theorem}

Both theorems will be proved in the remainder of this chapter.

\subsection{Constructing the Hodge module}\label{scn:HM}

From now on, we assume that the line bundle $B = \omYX \tensor \fu L^{-1}$ satisfies
the hypothesis in \eqref{eq:sections}.
For the sake of convenience, let $m \geq 1$ be the smallest integer with the property
that there is a nontrivial global section $s \in H^0 \bigl( Y, B^{\tensor m} \bigr)$.
Such a section defines a branched covering $\pi \colon Y_m \to Y$ of degree $m$,
unramified outside the divisor $Z(s)$; see \cite[\S3]{EV} for details. Since $m$ is
minimal, $Y_m$ is irreducible; let $\mu \colon Z \to Y_m$ be a resolution of
singularities that is an isomorphism over the complement of $Z(s)$, and define
$\varphi = \pi \circ \mu$ and $h = f \circ \varphi$. The following commutative
diagram shows all the relevant morphisms:
\begin{equation} \label{eq:geometry}
\begin{tikzcd}
Z \rar{\mu} \arrow[bend right=20]{drr}{h} \arrow[bend left=40]{rr}{\varphi} 
		& Y_m \rar{\pi} & Y \dar{f} \\
 & & X
\end{tikzcd}
\end{equation}
To simplify the notation, set $n = \dim X$ and $d = \dim Y = \dim Z$. Let $\shH^0 \hl
\QQ_Z^H \decal{d} \in \HM(X, d)$ be the polarizable Hodge module obtained by taking
the direct image of the
constant Hodge module on $Z$; restricted to the smooth locus of $h$, this is just the
polarizable variation of Hodge structure on the middle cohomology of the fibers. Let
$M \in \HM_X(X, d)$ be the summand with strict support $X$ in the decomposition of
$\shH^0 \hl \QQ_Z^H \decal{d}$ by strict support \cite[\S10]{sanya}. Let $\Mmod$
denote the underlying
regular holonomic left $\Dmod_X$-module, and $F_{\bullet} \Mmod$ its Hodge
filtration. Since $F_{\bullet} \Mmod$ is a good filtration, the associated graded
$\shA_X$-module
\[
	\gr_{\bullet}^F \! \Mmod = \bigoplus_{k \in \ZZ} \gr_k^F \! \Mmod
\]
is coherent over $\shA_X = \Sym \shT_X$; for simplicity, we denote the corresponding
coherent sheaf on the cotangent bundle by the symbol $\grFM$.

One has the following more concrete description of $\grFMb$. On $Z$, consider the
complex of graded $\hu \shA_X$-modules
\[
	C_{Z, \bullet} = \left\lbrack 
	\hu \shA_X^{\bullet-n} \tensor \bigwedge^d \shT_Z \to
	\hu \shA_X^{\bullet-n+1} \tensor \bigwedge^{d-1} \shT_Z \to \dotsb \to
	\hu \shA_X^{\bullet-n+d} \right\rbrack,
\]
placed in cohomological degrees $-d, \dotsc, 0$; the differential in the complex is
induced by the natural morphism $\shT_Z \to \hu \shT_X$.

\begin{proposition} \label{prop:Laumon}
In the category of graded $\shA_X$-modules, $\grFMb$ is a direct summand of
$R^0 \hl \bigl( \omZX \tensor C_{Z, \bullet} \bigr)$.
\end{proposition}

\begin{proof}
This is a special case of the following more general result: the complex 
\begin{equation} \label{eq:complex}
	\derR \hl \bigl( \omZX \tensor C_{Z, \bullet} \bigr)
\end{equation}
splits in the derived category of graded $\shA_X$-modules, and its $i$-th
cohomology module computes the associated graded of the Hodge module $\shH^i 
\hl \QQ_Z^H \decal{d}$. The proof is an application of several results by Saito. The
underlying filtered $\Dmod$-module of the trivial Hodge module is $(\OZ, F_{\bullet}
\OZ)$, where the filtration is such that $\gr_k^F \OZ = 0$ for $k \neq 0$. As shown
in \cite[Theorem~2.9]{mhmgv}, the direct image of $(\OZ, F_{\bullet} \OZ)$ in the
derived category of filtered $\Dmod$-modules satisfies
\begin{equation} \label{eq:Laumon}
	\gr_{\bullet}^F \hp(\OZ, F_{\bullet} \OZ) \simeq
		\derR \hl \left( \omZX \tensor_{\OZ} \gr_{\bullet+d-n}^F \! \OZ 
		\Ltensor_{\shA_Z} \hu \shA_X \right).
\end{equation}
Since the morphism $h \colon Z \to X$ is projective, the complex $\hp(\OZ,
F_{\bullet} \OZ)$ is strict and splits in the derived category \cite[\S16]{sanya};
the same is therefore true for the complex of
graded $\shA_X$-modules on the left-hand side of \eqref{eq:Laumon}. To conclude the
proof, we only have to show that the complex on the right-hand side of
\eqref{eq:Laumon} is quasi-isomorphic to the one in \eqref{eq:complex}. Now the
associated graded of the constant Hodge module is $\OZ$, in degree zero, with the
trivial action by $\shA_Z$. It is naturally resolved by the complex of graded
$\shA_Z$-modules
\begin{equation} \label{eq:resolution}
	\left\lbrack
	\shA_Z^{\bullet-d} \tensor \bigwedge^d \shT_Z \to
	\shA_Z^{\bullet-d+1} \tensor \bigwedge^{d-1} \shT_Z \to \dotsb \to
	\shA_Z^{\bullet} \right\rbrack,
\end{equation}
placed in cohomological degrees $-d, \dotsc, 0$. After shifting the grading by $d-n$
and tensoring over $\shA_Z$ by $\hu \shA_X$, we obtain the desired result.
\end{proof}

The lemma gives some information about the individual $\OX$-modules $\gr_k^F \Mmod$.

\begin{corollary} \label{cor:shGMmin}
One has $\gr_k^F \Mmod = 0$ for $k < n-d$, whereas
$\gr_{n-d}^F \Mmod \simeq \hl \omZX$.
\end{corollary}

\begin{proof}
The first assertion is clear because $C_{Z,k} = 0$ for $k < n-d$. To prove the second
assertion, recall that we have a canonical decomposition
\[
	\shH^0 \hl \QQ_Z^H \decal{d} \simeq M \oplus M',
\]
where $M$ has strict support $X$, and $M'$ is supported in a union of proper
subvarieties. \propositionref{prop:Laumon} shows that
\[
	F_{n-d} \Mmod \oplus F_{n-d} \Mmod' 
		\simeq R^0 \hl \bigl( \omZX \tensor C_{Z, n-d} \bigr) \simeq \hl \omZX.
\]
But now $F_{n-d} \Mmod'$ is supported in a union of proper subvarieties, whereas
$\hl\omZX$ is torsion-free; the conclusion is that $F_{n-d} \Mmod' = 0$. This is a
special case of a much more general result by Saito \cite[Proposition~2.6]{Saito-KC}.
\end{proof}

The complex $C_{Z, \bullet}$ is closely related to the set of singular cotangent
vectors $S_h \subseteq T^{\ast} X$ of the morphism $h \colon Z \to X$. Recall the
following notation:
\begin{equation} \label{eq:diag-dh}
\begin{tikzcd}
Z \dar{h} & T^{\ast} X \times_X Z \lar[swap]{p_2} \dar{p_1} \rar{\mathit{dh}} & T^{\ast} Z \\
X & T^{\ast} X \lar[swap]{p}
\end{tikzcd}
\end{equation}
Let us denote by $C_Z$ the complex of coherent sheaves on $T^{\ast} X \times_X Z$
associated with the complex of graded $\hu \shA_X$-modules $C_{Z, \bullet}$. 

\begin{proposition} \label{prop:support}
The support of $C_Z$ is equal to $\mathit{dh}^{-1}(0) \subseteq T^{\ast} X \times_X Z$.
\end{proposition}

\begin{proof}
The complex of graded $\shA_Z$-modules in \eqref{eq:resolution} is a resolution of
$\OZ$ as a graded $\shA_Z$-module, and so the associated complex of coherent sheaves
on $T^{\ast} Z$ is quasi-isomorphic to the structure sheaf of the zero section. The
proof of \propositionref{prop:Laumon} shows that $C_Z$ is the pullback of this complex via
the morphism $\mathit{dh}$ in the diagram in \eqref{eq:diag-dh};
its support must therefore be equal to $\mathit{dh}^{-1}(0)$.
\end{proof}

This result also implies the well-known fact that the characteristic variety of the
Hodge module $M$ is contained inside the set $S_h \subseteq T^{\ast} X$.

\begin{corollary}
The support of $\grFM$ is a union of irreducible components of $S_h$.
\end{corollary}

\begin{proof}
According to \propositionref{prop:Laumon}, one has
\[
\Supp	\grFM \subseteq \Supp R^0 p_{1\ast} \bigl( p_2^{\ast} \omZX \tensor C_Z \bigr),
\]
and because $\Supp C_Z$ is equal to $\mathit{dh}^{-1}(0)$, it follows that $\Supp
\grFM$ is contained in $S_h = p_1 \bigl( \mathit{dh}^{-1}(0) \bigr)$. Now
the support of $\grFM$ is by definition the characteristic variety of the regular
holonomic $\Dmod$-module $\Mmod$, and therefore of pure dimension $n = \dim X$. It
must therefore be a union of irreducible components of $S_h$, because we know from
\lemmaref{lem:Sf} that $\dim S_h \leq n$.
\end{proof}

Note that $S_h$ may very well have additional components of dimension $n$ that are
not accounted for by the Hodge module $M$. In any case, the existence of $M$ by
itself tells us nothing about the original morphism $f$.

\subsection{Constructing the graded module}

We now explain how to use the geometry of the branched covering in
\eqref{eq:geometry} to construct a graded $\shA_X$-submodule 
\[
	\shGb \subseteq \grFMb
\]
which, unlike $\grFMb$ itself, encodes information about the $f$-singular locus $D_f$
of the \emph{original} morphism $f \colon Y \to X$. 
In fact, the support of $\shGb$ will be contained in the set of singular cotangent vectors $S_f$;  the point is that $S_f$ is typically
much smaller than $S_h$, because both the covering and its resolution create
additional singular fibers. This construction will allow us to use positivity properties of Hodge modules towards the study of $D_f$.

By construction, the branched covering $Y_m$ is embedded into the total space of the
line bundle $B = \omYX \tensor \fu L^{-1}$, and so the pullback $\varphiu B$ has a
tautological section; the induced morphism $\varphiu B^{-1} \to \OZ$ is an
isomorphism over the complement of $Z(s)$. After composing with $\varphiu \OmY^k
\to \OmZ^k$, we obtain for every $k = 0, 1, \dotsc, d$ an injective morphism
\begin{equation} \label{eq:morphism}
	i_k \colon \varphiu \bigl( B^{-1} \tensor \OmY^k \bigr) \to \OmZ^k,
\end{equation}
that is actually an isomorphism over the complement of $Z(s)$. 

\begin{proposition} \label{prop:morphism}
There is a morphism of complexes of graded $\shA_X$-modules
\[
	\derR \fl \bigl( B^{-1} \tensor \omYX \tensor C_{Y, \bullet} \bigr) \to 
		\derR \hl \bigl( \omZX \tensor C_{Z, \bullet} \bigr),
\]
induced by the individual morphisms in \eqref{eq:morphism}.
\end{proposition}

\begin{proof}
By adjunction, it suffices to construct a morphism of complexes
\[
	\varphiu \bigl( B^{-1} \tensor \omYX \tensor C_{Y, \bullet} \bigr) \to
		\omZX \tensor C_{Z, \bullet}.
\]
Using the fact that $\shT_Y$ and $\OmY^1$ are dual to each other, we have
\[
	B^{-1} \tensor \omYX \tensor C_{Y, \bullet}^{k-d} \simeq
		\fu \omX^{-1} \tensor \fu \shA_X^{\bullet-n-k} \tensor B^{-1} \tensor \OmY^k,
\]
which gives us a natural isomorphism
\[
	\varphiu \bigl( B^{-1} \tensor \omYX \tensor C_{Y, \bullet}^{k-d} \bigr) \simeq
		\hu \omX^{-1} \tensor \hu \shA_X^{\bullet-n-k} 
			\tensor \varphiu \bigl( B^{-1} \tensor \OmY^k \bigr).
\]
By composing with \eqref{eq:morphism}, we obtain an $\hu \shA_X$-linear morphism to
\[
	\hu \omX^{-1} \tensor \hu \shA_X^{\bullet-n-k} \tensor \OmZ^k \simeq
	\omZX \tensor C_{Z, \bullet}^{k-d}.
\]
It remains to verify that the individual morphisms are compatible with the
differentials in the two complexes. Since they are by construction $\hu
\shA_X$-linear, the problem is reduced to proving the commutativity of the diagram
\[
\begin{tikzcd}[column sep=large]
\varphiu \bigl( B^{-1} \tensor \OmY^k \bigr) \dar \rar{i_k} & \OmZ^k \dar \\
\varphiu \bigl( B^{-1} \tensor \OmY^{k+1} \bigr) \tensor \hu \shT_X 
		\rar{i_{k+1} \tensor \id} & \OmZ^{k+1} \tensor \hu \shT_X
\end{tikzcd}
\]
in which the vertical morphisms are induced respectively by $\shT_Y \to \fu \shT_X$
and $\shT_Z \to \hu \shT_X$. This is an easy exercise.
\end{proof}

Now we can construct a graded $\shA_X$-module $\shGb$ in the following manner. By
composing the morphism from \propositionref{prop:morphism} with the projection to $\grFMb$,
we obtain a morphism of graded $\shA_X$-modules
\begin{equation} \label{eq:image}		
	R^0 \fl \bigl( B^{-1} \tensor \omYX \tensor C_{Y, \bullet} \bigr) 
	\to R^0 \hl \bigl( \omZX \tensor C_{Z, \bullet} \bigr)
	\to \grFMb.
\end{equation}
We then define $\shGb \subseteq \grFMb$ as the image of this morphism, in
the category of graded $\shA_X$-modules. Remembering that $B^{-1} \tensor \omYX = \fu
L$, we see that $\shGb$ is also a quotient of the graded $\shA_X$-module $L \tensor
R^0 \fl C_{Y, \bullet}$. We can use this observation to prove that $\shGb$ is
nontrivial.

\begin{proposition} \label{prop:shGmin}
One has $\shG_k = 0$ for $k < n-d$, whereas $\shG_{n-d} \simeq L \tensor \fl \OY$.
\end{proposition}

\begin{proof}
We make use of \corollaryref{cor:shGMmin}. The first assertion is clear because $\gr_k^F \Mmod = 0$ for $k < n-d$. By
construction, $\shG_{n-d}$ is a quotient of the $\OX$-module 
\[
	L \tensor R^0 \fl C_{Y,n-d} \simeq L \tensor \fl \OY.
\]
Since we already know that $\gr_{n-d}^F \Mmod$ is isomorphic to $\hl \omZX$, it is
therefore enough to prove that the morphism in \propositionref{prop:morphism} is injective
in degree $n-d$. The morphism in question is
\[
	\fl \bigl( B^{-1} \tensor \omYX \bigr) \to \hl \omZX,
\]
and is induced by \eqref{eq:morphism} for $k = d$. Now $i_d \colon \varphiu
\bigl( B^{-1} \tensor \omY \bigr) \to \omZ$ is injective, and because $\OY$ injects
into $\varphil \OZ$, the adjoint morphism $B^{-1} \tensor \omY \to \varphil \omZ$
remains injective. The second assertion follows from this because $\fl$ is left-exact.
\end{proof}

It is also not hard to show that the support of the associated coherent sheaf $\shG$
on the cotangent bundle is contained in the set $S_f$.

\begin{proposition} \label{prop:SuppG}
We have $\Supp \shG \subseteq S_f$.
\end{proposition}

\begin{proof}
By construction, $\shG$ is a quotient of the coherent sheaf $\pu L \tensor R^0 p_{1\ast}
C_Y$. But the complex $C_Y$ is supported on the set $\df^{-1}(0)$ by
\propositionref{prop:support}, and so the support of $\shG$ is contained in $S_f = p_1
\bigl( \df^{-1}(0) \bigr)$.
\end{proof}

\subsection{Additional properties}

Except in trivial cases, $\shGb$ is not itself the associated graded of a Hodge
module. Nevertheless, we shall see in this section that $\shGb$ inherits several good
properties from $\grFMb$.

\begin{lemma}
Every irreducible component of $\Supp \shG$ is the conormal variety of some
subvariety of $X$.
\end{lemma}

\begin{proof}
By a theorem of Saito \cite[\S29]{sanya}, $\grFM$ is a
Cohen-Macaulay sheaf on $T^{\ast} X$ of dimension $n = \dim X$;
in particular, it is unmixed, and every associated subvariety of $\grFM$ has
dimension $n$. This property is inherited by the subsheaf $\shG \subseteq
\grFM$; in particular, every irreducible component of $\Supp \shG$ is $n$-dimensional.
Since $\Supp \shG \subseteq S_f$, we conclude from \lemmaref{lem:Sf} that every such
component is the conormal variety of some subvariety of $X$.
\end{proof}

Recall our notation $D_f \subseteq X$ for the singular locus of the surjective
morphism $f \colon Y \to X$. Being part of a Hodge module, the coherent sheaves
$\gr_k^F \Mmod$ are locally free on the open subset $X \setminus D_h$ where $M$ is a
variation of Hodge structure. Surprisingly, the sheaves $\shG_k$ are torsion-free on
the much larger open set $X \setminus D_f$.

\begin{proposition}\label{prop:torsion}
For every $k \in \ZZ$, the sheaf $\shG_k$ is torsion-free on $X \setminus D_f$.
\end{proposition}

\begin{proof}
After replacing $X$ by the open subset $X \setminus D_f$, we may assume that
$\Supp \shG$ is contained in the zero section; the reason is that $\Supp \shG
\subseteq S_f$, and that $D_f$ is the image of $S_f$ minus the zero section. What
we need to prove is that $\shG_k$ is a torsion-free sheaf on $X$. This is equivalent
to saying that
\[
	\codim_X \Supp R^i \shHom_{\shO}(\shG_k, \omX) \geq i+1,
\]
for every $i \geq 1$; see for instance \cite[Lemma A.5]{PP}. We can compute the dual of $\shG_k$ directly by applying
Grothendieck duality to the projection $p \colon T^{\ast} X \to X$. Note first that 
\[
	\pl \shG \simeq \bigoplus_{k \in \ZZ} \shG_k
\]
is $\OX$-coherent because the support of $\shG$ is contained in the zero section; in
particular, $\shG_k = 0$ for $k \gg 0$. Since the relative dualizing sheaf is 
$\pu \omX^{-1}$, Grothendieck duality gives us
\[
	\derR \shHom_{\shO} \bigl( \shGb, \omX \bigr)
		\simeq \pl \derR \shHom_{\shO} \bigl( \shG, \shO \decal{n} \bigr),
\]
and so $\Supp R^i \shHom_{\shO}(\shG_k, \omX)$ is contained in the image under
$p$ of
\[
	\Supp R^i \shHom_{\shO} \bigl( \shG, \shO \decal{n} \bigr).
\]
As the zero section has codimension $n$, this reduces the problem to proving that 
\begin{equation} \label{eq:codim}
	\codim_{T^{\ast} X} \Supp R^i \shHom_{\shO} \bigl( \shG, \shO \decal{n} \bigr) 
		\geq n+i+1
\end{equation}
for every $i \geq 1$. On $T^{\ast} X$, we have a short exact sequence of coherent sheaves
\[
	0 \to \shG \to \grFM \to \grFM/\shG \to 0,
\]
from which we get a distinguished triangle
\[
	\derR \shHom_{\shO} \bigl( \grFM/\shG, \shO \decal{n} \bigr) \to 
	\derR \shHom_{\shO} \bigl( \grFM, \shO \decal{n} \bigr) 
		\to \derR \shHom_{\shO} \bigl( \shG, \shO \decal{n} \bigr) \to \dotsb
\]
All three complexes are concentrated in nonnegative degrees, because the original
sheaves are supported on a subset of codimension $\geq n$; moreover, $\derR
\shHom_{\shO} \bigl( \grFM, \shO \decal{n} \bigr)$ is a sheaf, because $\grFM$
is Cohen-Macaulay of dimension $n$. We conclude that
\[
	R^i \shHom_{\shO} \bigl( \shG, \shO \decal{n} \bigr) 
		\simeq R^{i+1} \shHom_{\shO} \bigl( \grFM/\shG, \shO \decal{n} \bigr) 
\]
for $i \geq 1$. But $\grFM/\shG$ is a sheaf, and so the support of the
right-hand side has codimension at least $n+i+1$; this implies the desired inequality
\eqref{eq:codim}.
\end{proof}

\subsection{Proof of Theorem~\ref*{thm:VZ}}

In this section, we prove \theoremref{thm:VZ} by putting together the results about
the graded $\shA_X$-module $\shGb$ that we have established so far. To resolve the
minor discrepancy in the indexing, we replace $M \in \HM_X(X,d)$ by its Tate twist
$M(d-n) \in \HM_X(X,2n-d)$; this leaves the underlying regular holonomic
$\Dmod$-module $\Mmod$ unchanged, but replaces the filtration $F_{\bullet} \Mmod$ by
the shift $F_{\bullet+n-d} \Mmod$. Similarly, we replace $\shGb$ by the shift
$\shG_{\bullet+n-d}$. Then the assertions in \ref{en:VZd} and \ref{en:VZe} hold by
construction, and the assertion in \ref{en:VZb} follows from
\corollaryref{cor:shGMmin}. In \propositionref{prop:SuppG}, we showed that $\Supp \shG
\subseteq S_f$, which proves \ref{en:VZa}. The remaining assertion in \ref{en:VZc}
has been established in \propositionref{prop:torsion}.

\begin{note}
At this point, the reader interested in starting with a technically simpler proof of
\theoremref{thm:hyperbolicity} when $X$ is not uniruled can move directly to
\parref{big-HM} and from there to \parref{scn:non-uniruled}.
\end{note}

\subsection{Constructing the Higgs bundle}

From now on, we assume in addition that the hypothesis in \eqref{eq:NCD} is
satisfied -- recall that we can always arrange this by blowing up $X$. We also assume
for simplicity that our morphism $f \colon Y \to X$ has connected fibers; this
implies that $\shG_0 \simeq L \tensor \fl \OY \simeq L$ is a line bundle.  Denote by
$\shV$ the polarizable variation of Hodge structure obtained by restricting
$M$ to the open subset $X \setminus D$; for the sake of convenience, we shall
consider the Hodge filtration as an increasing filtration $F_{\bullet} \shV$, with
the convention that $F_k \shV = F^{-k} \shV$. 

Let $\shVt$ be the canonical meromorphic extension \cite[Proposition~II.2.18]{Deligne} of
the flat bundle $(\shV, \nabla)$, and denote by $\shVt^{\geq \alpha}$ and
$\shVt^{>\alpha}$ Deligne's canonical lattices with eigenvalues contained in the
intervals $[\alpha, \alpha + 1)$ and $(\alpha, \alpha+1]$, respectively \cite[Proposition~I.5.4]{Deligne}. The flat connection on $\shV$ extends
uniquely to a logarithmic connection
\begin{equation} \label{eq:nabla}
	\nabla \colon \shVg \to \OmX^1(\log D) \tensor \shVg.
\end{equation}
As a consequence of Schmid's nilpotent orbit theorem, the Hodge filtration
$F_{\bullet} \shV$ extends to a filtration of $\shVg$ with locally free subquotients;
see \cite[(3.10.7)]{Saito-MHM} for a discussion of this point. The extension is given
by
\begin{equation} \label{eq:Hodge}
	F_k \shVg = \shVg \cap \jl F_k \shV,
\end{equation}
where $j \colon X \setminus D \into X$ is the inclusion. On the associated graded
with respect to the Hodge filtration, the connection then induces an $\OX$-linear operator
\begin{equation} \label{eq:Higgs}
	\theta \colon \gr_{\bullet}^F \shVg 
		\to \OmX^1(\log D) \tensor \gr_{\bullet+1}^F \shVg
\end{equation}
with the property that $\theta \wedge \theta = 0$. Setting 
\[
	\shEb = \gr_{\bullet}^F \shVg,
\]
we therefore obtain the desired (graded logarithmic) Higgs bundle $(\shEb, \theta)$. Since
$F_k \shV = 0$ for $k < 0$, it is clear that $\shE_k = 0$ in the same range.

It turns out that we can find a copy of $\shEb$ inside the associated graded
$\grFMb$ of the Hodge module $M$, and that the Higgs field $\theta$
can also be recovered from the $\shA_X$-module structure.  The connection on
$\shV$ induces on the meromorphic extension $\shVt$ the structure of a left
$\Dmod$-module. The formulas for the minimal extension in \cite[3.10]{Saito-MHM}
show that $\shVg \subseteq \Mmod \subseteq \shVt$.

\begin{lemma}
For every $k \in \ZZ$, we have an inclusion $\shE_k = \gr_k^F \shVg \subseteq \gr_k^F
\Mmod$.  
\end{lemma}

\begin{proof}
We observe that $\shVg \cap F_k \Mmod = F_k \shVg$. Indeed, the construction of the
Hodge filtration on $\Mmod$ in \cite[(3.10.12)]{Saito-MHM} is such that one has
\[
	\shVt^{>-1} \cap \jl F_k \shV \subseteq F_k \Mmod \subseteq \jl F_k \shV,
\]
as subsheaves of $\shVt$. Now intersect with $\shVg$ and use \eqref{eq:Hodge} to get
the desired identity. The inclusion $\shVg \subseteq \Mmod$ induces an
inclusion $F_k \shVg \subseteq F_k \Mmod$, and the identity we have just proved
implies that $\gr_k^F \shVg \to \gr_k^F \Mmod$ stays injective.
\end{proof}

To recover the action of the Higgs field $\theta$, let us choose local coordinates
$x_1, \dotsc, x_n$ that are adapted to the normal crossing divisor $D$. We denote the
corresponding vector fields by the symbols $\partial_1, \dotsc, \partial_n$. Since
the connection in \eqref{eq:nabla} has logarithmic poles, the action by the vector
field $x_i \partial_i$ preserves the lattice $\shVg$, and therefore induces an
$\OX$-linear morphism 
\[
	x_i \partial_i \colon \gr_k^F \shVg \to \gr_{k+1}^F \shVg.
\]
Note that if $x_i = 0$ is not a component of $D$, the factor of $x_i$ is
not needed, because in that case, $\partial_i$ itself already maps $\gr_k^F \shVg$
into $\gr_{k+1}^F \shVg$.
Putting together the individual morphisms, we obtain an $\OX$-linear morphism
\begin{equation} \label{eq:Higgs-new}
	\gr_k^F \shVg \to \OmX^1(\log D) \tensor \gr_{k+1}^F \shVg,
	\quad s \mapsto \sum_{i=1}^n \frac{\mathit{dx}_i}{x_i} \tensor (x_i \partial_i s),
\end{equation}
which is exactly the Higgs field in \eqref{eq:Higgs}. 

\subsection{Constructing the Higgs sheaf}

We can now construct a collection of subsheaves $\shF_k \subseteq \shE_k$ by
intersecting $\shG_k$ and $\shE_k = \gr_k^F \shVg$ inside the larger coherent sheaf
$\gr_k^F \Mmod$. Since the intersection may not be reflexive, we actually define
\[
	\shF_k = (\shG_k \cap \shE_k)^{\vee\vee} \subseteq \shE_k
\]
as the reflexive hull of the intersection. Since $\shE_k = 0$ for $k < 0$, we
obviously have $\shF_k = 0$ in the same range. It is also not hard to see that 
$\shFb \subseteq \shEb$ is compatible with the action of the Higgs
field in \eqref{eq:Higgs}. Indeed, using the notation introduced before
\eqref{eq:Higgs-new}, the vector field $x_i \partial_i$ maps $\shF_k$ into
$\shF_{k+1}$, due to the fact that $\shGb$ is a graded $\shA_X$-submodule of
$\grFMb$. But this means that the Higgs field takes $\shF_k$ into the subsheaf
$\OmX^1(\log D) \tensor \shF_{k+1}$.

As in the work of Viehweg and Zuo, the key property is that the action of the Higgs
field on $\shFb$ only creates poles along the smaller divisor $D_f$.

\begin{proposition} \label{prop:Higgs}
The Higgs field maps $\shF_k$ into $\OmX^1(\log D_f) \tensor \shF_{k+1}$.
\end{proposition}

The proof exploits the relationship between $F_{\bullet} \Mmod$ and the V-filtration
with respect to locally defined holomorphic functions on $X$. We briefly review the
relevant properties; for a more careful discussion of how the V-filtration enters
into the definition of Hodge modules, see \cite[\S9--12]{sanya}. Suppose then that $t
\colon U \to
\CC$ is a non-constant holomorphic function on an open subset $U \subseteq X$, with
the property that $t^{-1}(0)$ is a smooth divisor; also suppose that we have a
holomorphic vector field $\partial_t$ such that $\lie{\partial_t}{t} = 1$. To keep
the notation simple, we denote the restriction of $M$ and $(\Mmod, F_{\bullet}
\Mmod)$ to the open set $U$ by the same symbols.

Being part of a Hodge module, the $\Dmod$-module $\Mmod$ admits a \define{rational
V-filtration} $V^{\bullet} \Mmod$ with respect to the function $t$. As we are working
with left $\Dmod$-modules, this is a decreasing filtration, discretely indexed by
$\alpha \in \QQ$, with the following properties:
\begin{enumerate}
\item One has $t \cdot V^{\alpha} \Mmod \subseteq V^{\alpha+1} \Mmod$, with
equality for $\alpha > -1$.
\item One has $\partial_t \cdot V^{\alpha} \Mmod \subseteq V^{\alpha-1} \Mmod$.
\item The operator $t \partial_t - \alpha$ acts nilpotently on $\gr_V^{\alpha}
\Mmod = V^{\alpha} \Mmod / V^{>\alpha} \Mmod$, where $V^{>\alpha} \Mmod$ is
defined as the union of those $V^{\beta} \Mmod$ with $\beta > \alpha$.
\item Each $V^{\alpha} \Mmod$ is coherent over $V^0 \Dmod_X$, which is defined as the
$\OX$-subalgebra of $\Dmod_X$ that preserves the ideal $t \OX \subseteq \OX$.
\end{enumerate}
More generally, every regular holonomic $\Dmod$-module with quasi-unipotent local
monodromy\footnote{This is a shorthand for saying that all eigenvalues of the
monodromy operator on the perverse sheaf of nearby cycles $\psi_t \DR(\Mmod)$ are
roots of unity.} around the divisor $t^{-1}(0)$ has a unique rational V-filtration; this
result is due to Kashiwara \cite{Kashiwara}. The uniqueness statement implies that if two
$\Dmod$-modules $\Mmod_1$ and $\Mmod_2$ admit rational V-filtrations, then any
morphism $f \colon \Mmod_1 \to \Mmod_2$ between them is strictly compatible with
these filtrations, in the sense that
\[
	f(\Mmod_1) \cap V^{\alpha} \Mmod_2 = f \bigl( V^{\alpha} \Mmod_1 \bigr)
\]
for all $\alpha \in \QQ$.

Saito defines the category of Hodge modules by requiring, among other
things, that the Hodge filtration $F_{\bullet} \Mmod$ interacts well with the
rational V-filtration $V^{\bullet} \Mmod$. The first requirement in the definition is
that
\begin{equation} \label{eq:t}
	t \colon F_k V^{\alpha} \Mmod \to F_k V^{\alpha+1} \Mmod
\end{equation}
must be an isomorphism for $\alpha > -1$; the second requirement is that
\begin{equation} \label{eq:partial}
	\partial_t \colon F_k \gr_V^{\alpha} \Mmod \to F_{k+1} \gr_V^{\alpha-1} \Mmod
\end{equation}
must be an isomorphism for $\alpha < 0$. Here and in what follows, $F_{\bullet}
\gr_V^{\alpha} \Mmod$ means the filtration induced by $F_{\bullet} \Mmod$; thus
\[
	F_k \gr_V^{\alpha} \Mmod = 
		\frac{V^{\alpha} \Mmod \cap F_k \Mmod + V^{>\alpha} \Mmod}{V^{>\alpha} \Mmod}.
\]
These two conditions together give us very precise information on how the two
operators $t$ and $\partial_t$ interact with the Hodge filtration $F_{\bullet}
\Mmod$.

\begin{proof}[Proof of \propositionref{prop:Higgs}]
Since $\shF_k$ and $\OmX^1(\log D_f) \tensor \shF_{k+1}$ are reflexive coherent
sheaves, we only need to prove the assertion outside a
subset of codimension $\geq 2$. After removing the singular locus of the normal
crossing divisor $D$, it is therefore enough to show that when the Higgs field in
\eqref{eq:Higgs-new} is applied to a local section of $\shF_k$, it does not actually
produce any poles
along the components of $D$ that do not belong to $D_f$. Fix such a component, and on a
sufficiently small open neighborhood $U$ of its generic point, choose local
coordinates $x_1, \dotsc, x_n$ such that $D \cap U$ is defined by the equation $x_n =
0$. Because we can ignore what happens on a subset of codimension $\geq 2$, we may
assume that $\shF_k = \shG_k \cap \shE_k$ on $U$. Moreover, the component in question
does not belong to $D_f$, and so we may further assume that $D_f \cap U = \emptyset$; 
the part of $S_f$ that lies over $U$ is then contained in the zero section of
the cotangent bundle. As $\shGb$ is coherent over $\shA_X$, this implies
that any section in $H^0(U, \shG_k)$ is annihilated by a sufficiently large power of
$\partial_n$.

Let $V_i^{\bullet} \Mmod$ be the rational V-filtration with respect to the
function $x_i$. Since $\Mmod$ is a flat bundle outside the divisor $x_n = 0$, it is
easy to see from the definition that 
\[
	V_i^{\alpha} \Mmod = x_i^{\max(0, \lfloor \alpha \rfloor)} \Mmod \subseteq \Mmod
\]
is essentially the $x_i$-adic filtration for $i = 1, \dotsc, n-1$; the same is true also
on the larger $\Dmod$-module $\shVt$. The defining property of the canonical lattices
implies that
\[
	\shVt^{\geq \alpha} = \bigcap_{i=1}^n V_i^{\alpha} \shVt
		= V_n^{\alpha} \shVt
\]
for $\alpha < 1$, as noted for example in \cite[(3.10.1)]{Saito-MHM}. Because any
morphism of $\Dmod$-modules is strictly compatible with the rational V-filtration, we obtain
\[
	V_n^{\alpha} \Mmod = \Mmod \cap V_n^{\alpha} \shVt 
		= \Mmod \cap \shVt^{\geq \alpha}
\]
as long as $\alpha < 1$; in particular, $V_n^0 \Mmod = \shVg$. Over the open set $U$, we
thus get 
\[
	\shE_k = \gr_k^F \shVg = V_n^0 \gr_k^F \Mmod \subseteq \gr_k^F \Mmod,
\]
where $V_n^{\bullet} \gr_k^F \Mmod$ again means the filtration induced by $V_n^{\bullet}
\Mmod$, that is to say,
\[
	V_n^{\alpha} \gr_k^F \Mmod = 
		\frac{V_n^{\alpha} \Mmod \cap F_k \Mmod + F_{k-1} \Mmod}{F_{k-1} \Mmod}.
\]

Now given any section $s \in \Gamma(U, \shF_k)$, we need to argue that $\partial_n s
\in H^0(U, \shF_{k+1})$; this will guarantee that the Higgs field in
\eqref{eq:Higgs-new} does not create a pole along $x_n = 0$ when applied to the
section $s$. Viewing $s$ as a section of the larger locally free sheaf $\shE_k = V_n^0
\gr_k^F \Mmod$, and remembering that the operator $\partial_n$ maps $F_k \Mmod$ into
$F_{k+1} \Mmod$ and $V_n^0 \Mmod$ into $V_n^{-1} \Mmod$, we obtain
\begin{equation} \label{eq:partial-s}
	\partial_n s \in H^0 \bigl( U, V_n^{-1} \gr_{k+1}^F \Mmod \bigr).
\end{equation}
We shall argue that, in fact, $\partial_n s \in H^0 \bigl( U, V_n^0 \gr_{k+1}^F \Mmod
\bigr)$. This is the crucial step in the proof; it rests entirely on the
compatibility between the Hodge filtration and the rational V-filtration, in the form
of Saito's condition \eqref{eq:partial}.

Since $\shF_k \subseteq \shG_k$, we already know that $\partial_n^{\ell+1} s = 0$ for
some $\ell \geq 1$; in other words, our section $\partial_n s$ is annihilated by the operator
$\partial_n^{\ell}$. To make use of this fact, consider the following commutative
diagram with short exact columns:
\[
\begin{tikzcd}
	F_k \gr_{V_n}^{\alpha} \Mmod \dar[hook]
		\rar{\partial_n^{\ell}} & F_{k+\ell} \gr_{V_n}^{\alpha-\ell} \Mmod \dar[hook] \\
	F_{k+1} \gr_{V_n}^{\alpha} \Mmod \dar[two heads]
		\rar{\partial_n^{\ell}} & F_{k+1+\ell} \gr_{V_n}^{\alpha-\ell} \Mmod \dar[two heads] \\
	\gr_{k+1}^F \gr_{V_n}^{\alpha} \Mmod
		\rar{\partial_n^{\ell}} & \gr_{k+1+\ell}^F \gr_{V_n}^{\alpha-\ell} \Mmod
\end{tikzcd}
\]
The condition in \eqref{eq:partial} relating $F_{\bullet} \Mmod$ and $V_n^{\bullet}
\Mmod$ tells us that the morphisms in the first and second row
are isomorphisms for $\alpha < 0$; of course, the morphism in the third row is then
also an isomorphism. Taking $\alpha = -1$, we see that
\[
	\partial_n^{\ell} \colon \gr_{k+1}^F \gr_{V_n}^{-1} \Mmod \to 
		\gr_{k+1+\ell}^F \gr_{V_n}^{-1-\ell} \Mmod
\]
is an isomorphism; because the image of $\partial_n s$ belongs to the kernel, we
conclude that $\partial_n s$ is in fact a section of 
\[
	V_n^{>-1} \gr_{k+1}^F \Mmod = V_n^{\alpha} \gr_{k+1}^F \Mmod
\]
for some $\alpha > -1$. As long as $\alpha < 0$, we can repeat this argument and
further increase the value of $\alpha$; because the rational V-filtration is discretely
indexed, we eventually arrive at the conclusion that
\[
	\partial_n s \in H^0 \bigl( U, V_n^0 \gr_{k+1}^F \Mmod \bigr)
		= H^0(U, \shE_{k+1}).
\]
Since also $\partial_n s \in H^0(U, \shG_{k+1})$, we obtain $\partial_n s \in H^0(U,
\shF_{k+1})$, as needed.
\end{proof}

We also need to have some information about the subsheaf $\shF_0 \subseteq F_0
\shVg$. From the definition of the Hodge filtration on $\Mmod$ in
\cite[(3.10.12)]{Saito-MHM}, we immediately get $F_0 \Mmod = F_0 \shVt^{>-1}$, and so
the problem is to compute the intersection
\[
	\shG_0 \cap F_0 \shVg \subseteq F_0 \shVt^{>-1}.
\]
This turns out to be fairly subtle, and the answer depends on the local monodromy
around the components of the divisor $D_f$. Fortunately, the following rather weak
result is enough for our purposes.

\begin{proposition} \label{prop:shFmin}
We have $L(-D_f) \subseteq \shF_0 \subseteq L$.
\end{proposition}

\begin{proof}
It is again enough to prove this outside a closed subset of codimension $\geq 2$.
After removing the singular locus of the normal crossing divisor $D$, we may
therefore consider one component of $D$ at a time; as in the proof of
\propositionref{prop:Higgs}, we choose local coordinates $x_1, \dotsc, x_n$ on an
open subset $U$ such that $D \cap U$ is defined by the equation $x_n = 0$. Let $s \in
H^0(U, \shG_0)$ be any section. By definition,
\[
	s \in H^0(U, F_0 \shVt^{>-1}) = H^0(U, F_0 V_n^{>-1} \Mmod).
\]
Now there are two possibilities. If the component in question belongs to the divisor
$D_f$, we use the obvious fact that
\[
	x_n s \in H^0(U, F_0 V_n^0 \Mmod) \subseteq H^0(U, \shF_0)
\]
to conclude that multiplication by a local equation for $D_f$ maps $\shG_0 \simeq L$
into the subsheaf $\shF_0$. If the component in question does not belong to the
divisor $D_f$, then we argue as in the proof of \propositionref{prop:Higgs}. Namely,
the part of $S_f$ that lies over $U$ is contained in the zero section of the
cotangent bundle, which means that we have $\partial_n^{\ell} s = 0$ for $\ell \gg
0$. As before, we use \eqref{eq:partial} to conclude that 
\[
	s \in	H^0(U, F_0 V_n^0 \Mmod) = H^0(U, \shF_0),
\]
which leads to the desired conclusion also in this case.
\end{proof}

\subsection{Proof of Theorem~\ref*{thm:Higgs}}

We finish this chapter by proving \theoremref{thm:Higgs}. As already
mentioned, the (graded logarithmic) Higgs bundle $\shEb = \gr_{\bullet}^F \shVg$, with the
Higgs field in \eqref{eq:Higgs}, comes from the polarizable variation of Hodge
structure $\shV$ on $X \setminus D$, and so \ref{en:Higgse} is true by construction.
The graded submodule $\shFb \subseteq \shEb$ satisfies \ref{en:Higgsc} by
construction, \ref{en:Higgsb} because of \propositionref{prop:shFmin}, and 
\ref{en:Higgsd} because of \propositionref{prop:Higgs}.

\section{Positivity for Hodge modules and Higgs bundles}\label{big}

\subsection{Background on weak positivity}

In this paragraph we fix a smooth quasi-projective variety $X$, and a torsion-free coherent sheaf $\shF$ on $X$. 

\begin{definition}[\cite{Viehweg1, Viehweg2}] \hfill
\begin{renumerate}
\item We call $\shF$ \emph{weakly positive over an open set $U \subseteq X$} if for
every integer $\alpha > 0$ and every ample line bundle $H$ on $X$, there is an integer $\beta > 0$ such that 
$$(S^{\alpha \beta} \shF)^{\vee \vee} \otimes H^{\otimes \beta}$$
is generated by global sections at each point of $U$. We say that $\shF$ is
\emph{weakly positive} if such an open set $U \neq \emptyset$ exists. 
\item We call $\shF$ \emph{big} (in the sense of Viehweg) if for any line bundle $L$ on $X$, there exists some 
integer $\gamma > 0$ such that $(S^{\gamma} \shF)^{\vee \vee} \otimes L^{-1}$ is weakly positive.
\end{renumerate}
\end{definition}

We recall some basic facts needed in the next section; they are immediate applications of \cite[Lemma 1.4]{Viehweg1} and 
\cite[Lemma 3.6]{Viehweg2}.

\begin{lemma}\label{lemma:basic-WP}
Let $\shF$ and $\shG$ be torsion-free coherent sheaves on $X$. Then:
\begin{enumerate}
\item If $\shF \rightarrow \shG$ is surjective over $U$, and if $\shF$ is weakly positive over $U$, then $\shG$ is weakly positive over $U$. Moreover, if $\shF$ is big, then $\shG$ is big.
\item If $\shF$ is weakly positive and $A$ is a big line bundle, then $\shF \otimes A$ is big.
\item If $\shF$ is big, then $\det \shF$ is a big line bundle.
\end{enumerate}
\end{lemma}

\subsection{Positivity for Hodge modules}\label{big-HM}
Let $M$ be a pure Hodge module on a smooth projective variety $X$, with underlying filtered $\Dmod_X$-module
$(\Mmod, F_{\bullet} \Mmod)$.
For each $k$, the filtration induces Kodaira-Spencer type $\OX$-module homomorphisms 
$$\theta_k:  \gr_k^F \Mmod \longrightarrow \gr_{k+1}^F \Mmod \otimes \OmX^1,$$
and we shall use the notation
$$K_k (M): = \ker \theta_k.$$
Below we will make use of the following weak positivity statement, extending to Hodge modules results of \cite{Zuo} and \cite{Brunebarbe}, which are themselves generalizations of the well-known Fujita-Kawamata semipositivity theorem.

\begin{theorem}[{\cite[Theorem A]{PW}}]\label{thm:WP}
If $M$ is a pure polarizable Hodge module with strict support $X$, then the reflexive sheaf $K_k (M)^\vee$ is weakly positive for any $k$.
\end{theorem}

We now give an ad-hoc definition, for repeated use in what follows. Recall that 
given filtered $\Dmod_X$-module $(\Mmod, F_{\bullet} \Mmod)$, the associated graded $\gr_{\bullet}^F \Mmod$ is coherent 
graded $\shA_X$-module, with $\shA_X = \Sym \shT_X$. Moreover, we denote 
\begin{equation}\label{minimum}
p (\Mmod) : = \min ~\{ ~p ~|~F_p \Mmod \neq 0~\}.
\end{equation}
In other words, 
$$F_{p(\Mmod)} \Mmod = \gr_{p(\Mmod)} ^F \Mmod$$ 
is the lowest nonzero graded piece in the filtration on $\Mmod$. If $(\Mmod,
F_{\bullet} \Mmod)$ underlies a Hodge module $M$, we also use the notation $p(M)$
instead of $p(\Mmod)$.

\begin{definition}[Large graded $\shA_X$-submodules]
Let $M$ be a pure Hodge module with strict support $X$.  A graded $\shA_X$-submodule
$$\shGb \subseteq \gr_{\bullet}^F \Mmod$$
is called \emph{large (with respect to $D$)} if there exist a big line bundle $A$ and an effective divisor $D$ on $X$, together with an integer $\ell \ge 0$, such that:
\begin{itemize}
\item there is a sheaf inclusion $A (-\ell D) \hookrightarrow \shG_{p(M)} $.
\item the support of the torsion of all $\shG_k$ is contained in $D$.\footnote{Typically $D$ may be the complement of the locus where $M$ is a variation of Hodge structure, though we will also 
have to deal with the case when $D$ is strictly contained in that locus.}
\end{itemize}
\end{definition}

\medskip

For applications we need to prove the following stronger version of
\theoremref{thm:HM-positivity}; the reason is that the locus where the Hodge module
we consider is not a variation of Hodge structure is usually bigger than the
singular locus of the family we are interested in.

\begin{theorem}\label{thm:submodule-positivity}
Let $X$ be a smooth projective variety, and let $M$ be a pure Hodge module $M$ with strict support $X$ and underlying filtered
$\Dmod_X$-module $(\Mmod, F_{\bullet} \Mmod)$, which is generically a variation of
Hodge structure of weight $k$. Assume that there exists a graded $\shA_X$-submodule
$\shGb \subseteq \gr_{\bullet}^F \Mmod$ which is large with respect to a divisor $D$.
Then at least one of the following holds:
\begin{renumerate}
\item $D$ is big.
\item There exist $1 \le s \le k$, $r \ge 1$, and a big coherent sheaf $\shH$ on $X$,
such that  $$\shH \hookrightarrow (\OmX^1)^{\otimes s} \otimes \OX(rD)$$
\end{renumerate}
Moreover, if $X$ is not uniruled, then $\omX (D)$ is big.
\end{theorem}
\begin{proof}
Note to begin with that $F_{p(M)}  \Mmod$ is a torsion-free sheaf; see for instance \cite{Saito-KC}. It follows that $\shG_{p(M)}$ is 
torsion-free as well. By assumption, there is a big line bundle $A$ on $X$ and an integer $\ell \ge 0$, 
together with an injective sheaf morphism 
$$A (- \ell D)  \hookrightarrow \shG_{p(M)}.$$

Denoting for simplicity $p = p(M)$, the graded $\shA_X$-module structure induces a chain of homomorphisms of coherent 
$\OX$-modules
$$0 \longrightarrow \shG_p \overset{\theta_p}{\longrightarrow}  \shG_{p+1}  \otimes  \OmX^1 
\overset{\theta_{p+1} \circ \id}{\longrightarrow}  \shG_{p+2}  \otimes  
(\OmX^1)^{\otimes 2} \longrightarrow \cdots$$
Just as with Hodge modules, we will denote
$$K_k = K_k (\shGb) : = \ker \big( \theta_k:  \shG_k  \longrightarrow \shG_{k+1}  \otimes \OmX^1\big).$$
There are obvious inclusions 
$$K_k \hookrightarrow K_k (M),$$
which by \theoremref{thm:WP} and \lemmaref{lemma:basic-WP}(1) imply that $K_k^\vee$ are weakly positive for all $k$.
To start making use of this property, note to begin with that given the inclusion of $A ( - \ell D)$ into $\shG_p$, there are 
two possibilities: 

The first is that the induced homomorphism 
$$A ( -\ell D) \longrightarrow \shG_{p+1} \otimes  \OmX^1$$
is not injective, i.e. $A$ maps into the torsion of this sheaf, whose support is assumed to be contained in $D$. It follows that there exists a 
non-trivial subscheme $Z \subset X$ such that $Z_{\mathrm{red}} \subseteq D$ and $A (-\ell D)  \otimes {\mathcal I}_Z \subset K_p$. This implies that there exists an integer $r \ge 1$ and an inclusion $A (- rD) \hookrightarrow K_p$, which induces a non-trivial homomorphism
$$K_p^\vee \longrightarrow A^{-1} (rD).$$
Using  the weak positivity of $K_p^\vee$ again, we get that $A^{-1} (rD)$ is pseudo-effective. Since $A$ is big, we get that $D$ must be big, i.e. the condition in (i).

The second possibility is that we have an inclusion
$$A (-\ell D) \hookrightarrow   \shG_{p+1} \otimes  \OmX^1.$$
We can then repeat the same argument via the morphisms $\theta_s \circ \id$ with $s \ge p+1$. 
The next thing to note however is that there is an $s \le k$ where the inclusions will have to stop, i.e. such that 
$$A \subseteq \shG_{p+s}  \otimes  (\OmX^1)^{\otimes s} \,\,\,\,\,{\rm and} \,\,\,\,\,
A \not \subseteq  \shG_{p+s+1} \otimes  (\OmX^1)^{\otimes s+1}.$$
Indeed, note that an inclusion $A \subseteq  \gr_{p+t}^F \Mmod \otimes  (\OmX^1)^{\otimes t}$ can only hold as long as 
$\gr_{p+t}^F \Mmod$ is not a torsion sheaf. Recall however that $(\Mmod, F_{\bullet} \Mmod)$  underlies an extension of a
variation of Hodge structure ${\bf V}$ of weight $k$ on an open set $U \subset X$. Thus over $U$ the sheaves $\gr_{p+t}^F \Mmod$ coincide with Hodge bundles of ${\bf V}$, and therefore are non-zero only for $t \le k$.

As above,  this implies that there exists some $r \ge 1$ such that
\begin{equation}\label{limit_case}
A (-rD) \subseteq K_{p+s} \otimes  (\OmX^1)^{\otimes s} .
\end{equation}
We conclude that for $s$ as in ($\ref{limit_case}$), there exists a nontrivial homomorphism
$$K_{p+s}^\vee \otimes  A \longrightarrow (\OmX^1)^{\otimes s} \otimes \OX(rD)$$
and, using the weak positivity of $K_{p+s}^\vee$ and \lemmaref{lemma:basic-WP}(1) and (2), taking its image we obtain an inclusion
$$\shH \hookrightarrow (\OmX^1)^{\otimes s} \otimes \OX(rD)$$
with $\shH$ a big sheaf on $X$, i.e. the condition in (ii).

Let us now assume that $X$ is not uniruled. By \cite[Corollary 0.3]{BDPP} it follows that $\omX$ is pseudo-effective. If $D$ is big, then we immediately get the conclusion. Otherwise we employ the standard argument based on the pseudo-effectivity of quotients by Viehweg-Zuo type sheaves, inspired by an idea in \cite[\S2.2]{CP} (see also \cite{Patakfalvi}): using (ii), we have a short exact sequence
$$0 \longrightarrow \shH \longrightarrow (\OmX^1)^{\otimes s} \otimes \OX(rD) \longrightarrow \shQ \longrightarrow 0,$$
and passing to the saturation of $\shH$ we can assume that $\shQ$ is torsion-free.
Using that $X$ is not uniruled, a special case of \cite[Theorem 1.2]{CPa} says that every torsion-free quotient of 
$(\OmX^1)^{\otimes s}$ has pseudo-effective determinant; it follows that $\det \big( \shQ (-rD) \big)$ is pseudo-effective, which implies that $\det \shQ$ is pseudo-effective as well. Since $\shH$ is big, its determinant is also big by \lemmaref{lemma:basic-WP}(3), and one obtains by passing to determinants in the sequence above that 
$$\omX^{\otimes s} \otimes \OX(nrs D),$$ 
is big, with $n= \dim X$. Finally, $\omX$ is pseudo-effective, and so multiplying by its suitable power implies that $\omX(D)$ is big.
\end{proof}

\begin{example}\label{ex:uniruled-counterex}
The last statement of \theoremref{thm:submodule-positivity} may fail if $X$ is
uniruled. For example, consider the double covering of $\PP^1$ branched at the two
points $0$ and $\infty$. Let $M$ be the direct image of the constant Hodge module,
and denote by $(\Mmod, F_{\bullet} \Mmod)$ the underlying filtered $\Dmod$-module.
Then \corollaryref{cor:shGMmin} shows that $F_0 \Mmod$ is the direct image of the
relative canonical bundle, which equals $\shO_{\PP^1} \oplus \shO_{\PP^1}(1)$. Even
though $F_0 \Mmod$ contains an ample line bundle, $\PP^1 \setminus \{0, \infty\}$ is
of course not of log general type. This phenomenon is partly explained by
\propositionref{prop:shFmin}: in the process of constructing the Higgs subsheaf
$\shFb$, the fact that the local system corresponding to the summand
$\shO_{\PP^1}(1)$ has nontrivial monodromy of order $2$ around each of the two
points means that we end up with $\shF_0 = \shO_{\PP^1}(-1)$, which is no longer
ample.
\end{example}

\subsection{Positivity for Higgs bundles}\label{big-Higgs}
We will also need a version of \theoremref{thm:submodule-positivity} in the case of
(graded logarithmic) Higgs bundles. This will 
allow us later on to deal with the a priori possibility of $X$ being uniruled.

The set-up is as follows: $X$ is a smooth projective variety, and $D$ a simple normal crossings divisor on $X$. 
We consider a (graded logarithmic) Higgs bundle
\[
	\theta_p \colon \shE_p \to \OmX^1(\log D) \tensor \shE_{p+1}
\]
extending a polarizable variation of Hodge structure of weight $\ell$ on $X \setminus D$; up to Tate twist, we can make the convention that $\shE_p \neq 0$ if only if $0 \le p \le \ell$. We also consider a graded submodule 
$\shFb \subset \shEb$ having the property that
$$\theta_p (\shF_p) \subseteq \OmX^1(\log B) \tensor \shF_{p+1}$$
for some divisor $B \subseteq D$. Note that since $\shE_p$ are vector bundles, the sheaves $\shF_p$ are automatically torsion-free. By analogy with the previous section, we say that $\shFb$ is \emph{large} if there exists 
a big line bundle $A$ such that $A\subseteq \shF_0$.

The first part of the next theorem is essentially due to Viehweg-Zuo \cite{VZ2}, at least in the geometric case; it can be proved 
completely analogously to \theoremref{thm:submodule-positivity}, replacing the chain of 
coherent $\shO_X$-module homomorphisms there with 
 $$0 \longrightarrow \shF_0 \overset{\theta_0}{\longrightarrow}  \shF_{1}  \otimes  \OmX^1 (\log B)
\overset{\theta_{1} \circ \id}{\longrightarrow}  \shF_{2}  \otimes  
\big(\OmX^1 (\log B)\big)^{\otimes 2} \longrightarrow \cdots$$
The argument is in fact simpler, as no torsion issues arise. The weak positivity of $K_k (\shEb)^\vee$, with
 $$K_k (\shEb) : = \ker \big( \theta_k:  \shE_k  \longrightarrow \OmX^1 (\log D) \otimes \shE _{k+1} \big),$$
is deduced in \cite{PW} (see Theorem 4.9 and its proof, a step towards the proof of \theoremref{thm:WP}) 
as a quick corollary of the results of \cite{Zuo} and \cite{Brunebarbe}.

\begin{theorem}\label{thm:submodule-Higgs}
Assume that $X$ is endowed with large submodule of a (graded logarithmic) Higgs bundle, as above. 
Then there exist a big coherent sheaf $\shH$ on $X$ and an integer $1 \le s \le \ell$, together with an inclusion
$$\shH \hookrightarrow \big(\OmX^1(\log B)\big)^{\otimes s}.$$
In particular, $(X, B)$ is of  log-general type, i.e. $\omega_X(B)$ is big.
\end{theorem}

The last part of the theorem is due to Campana-P\u{a}un; once we have the existence of a Viehweg-Zuo sheaf $\shH$ as in the statement, it follows from:

\begin{theorem}[{\cite[Theorem 7.6]{CPa}}]\label{thm:CP}
Let $X$ be a smooth projective variety and $B$ a simple normal crossings divisor on $X$. Assume that there for some $s \ge 1$ 
there is an inclusion 
$$\shH \hookrightarrow \big(\OmX^1(\log B)\big)^{\otimes s},$$
where $\shH$ is a sheaf whose determinant is big. Then $\omega_X (B)$ is a big line bundle.
\end{theorem} 

Note that in \cite{CPa} the result is stated when $\shH$ is a line bundle, but the
proof works identically for any subsheaf such that 
$\det \shH$ is big. Moreover, what we state here is only a special case of their theorem; in fact, the possible non-pseudoeffectivity of 
$\omega_X$ is very cleverly dealt with in \cite{CPa} by proving a more general theorem that applies to the orbifold setting as well.

\section{Families of varieties}

To show the statement in \theoremref{thm:hyperbolicity}, it is immediate that after a birational modification we can assume that the $f$-singular locus is a divisor $D_f$, and hence it is also enough to just take $D = D_f$. We will always assume that this is the case in what follows.

\subsection{Non-uniruled case}\label{scn:non-uniruled}
In this section we prove  \theoremref{thm:hyperbolicity} under the extra assumption that the base space $X$ is not uniruled. This is 
included since at a first reading it allows to avoid many of the technicalities involved in dealing with the remaining case, while containing all the key ideas. The proof in the general case is given in the next section.

As recalled in the introduction, Viehweg's $Q_{n,m}$ conjecture states that if $f$ is a fiber space with maximal variation, 
then $\det \fl \omYX^{\otimes m}$ is big for some $m > 0$. It is known to hold when the fibers are of general type by \cite{Kollar}, and more generally (nowadays) when they have good minimal models by \cite{Kawamata}. Thus \theoremref{thm:hyperbolicity} in the non-uniruled case is a consequence of the following statement and \theoremref{thm:submodule-positivity}.

\begin{theorem}\label{thm:submodule-existence}
Let $f: Y \rightarrow X$ be an algebraic fiber space between smooth projective varieties, with branch locus a divisor 
$D_f \subset X$. Assume that there is an integer $m > 0$ such that $\det \fl \omYX^{\otimes m}$ is big. Then there exists a pure Hodge module $M$ with strict support $X$  and underlying filtered
$\Dmod_X$-module $(\Mmod, F_{\bullet} \Mmod)$, together with a graded $\shA_X$-submodule 
$\shGb \subseteq \gr_{\bullet}^F \Mmod$ which is large with respect to $D_f$.
\end{theorem}

\begin{proof}
Note that since the conclusion is purely on $X$, we are allowed to change the domain $Y$ as necessary.

\medskip
\noindent
{\bf Step 1.}
We reduce to the following situation: given a fiber space over $X$ as in the statement, and given any ample line bundle $A$ on $X$, we can modify the picture to a new family $f': Y' \rightarrow X$ satisfying the property that there exists an integer $k_0 \ge 0$ such that:
$$B = \omega_{Y'/X} \otimes {f'}^*L^{-1} \,\,\,\,{\rm satisfies} \,\, (\ref{eq:sections}) \,\,{\rm  with} \,\,\,\, L = A (-k_0 D_f).$$

To this end, fix an $m > 0$ such that
$$L_m : = \det \fl \omYX^{\otimes m}$$  
is a big line bundle. Given any ample line bundle $M$ on $X$, we will produce a new family 
$f': Y' \rightarrow X$, smooth over $U = X \setminus D_f$, such that
\begin{equation}\label{wish}
H^0 \big(Y', \omega_{Y'/X}^{\otimes m} \otimes {f'}^* M^{-1}(kD_f)\big) \neq 0
\end{equation}
for some integer $k \ge 0$. 
In particular, we can take $M = A^{\otimes m}$; also, perhaps by increasing it, we can assume $k = k_0 \cdot m$ for some $k_0\ge 0$ in order  to obtain the reduction step. 

To prove the existence of such a family $f'$, note first that for $N$ sufficiently large we can write
$$L_m^{\otimes N} \simeq M \otimes \OX(B),$$
where $B$ is an effective divisor.  Denote by $r_0$ the rank of $\fl \omYX^{\otimes m}$ over $U$, where it is a locally free sheaf by Siu's invariance of plurigenera, and define $r := N\cdot r_0$. Then there is an inclusion of sheaves 
$$L_m^{\otimes N} \hookrightarrow (\fl \omYX^{\otimes m})^{\otimes r}$$
(which is split over the locus where $\fl \omYX^{\otimes m}$
is locally free).

Now we take advantage of Viehweg's fiber product trick. Consider a resolution of
singularities $Y^{(r)}$ of the main component of the $r$-fold fiber
product $Y \times_{X}  \cdots \times_{X} Y$, with its induced morphism
$f^{(r)}: Y^{(r)} \rightarrow X$. Note that this morphism is smooth over $U$ as well; 
moreover, it is well known (see \cite[Corollary 4.11]{Mori} or \cite[Lemma 3.5]{Viehweg1}) that there exists a morphism
$$f^{(r)}_* \omega_{Y^{(r)}/X}^{\otimes m} \longrightarrow \big( (\fl \omYX^{\otimes m})^{\otimes r}\big)^{\vee \vee},$$
which is an isomorphism over $U$. Since $L_m^{\otimes N}$ injects into the right hand side, and the 
morphism $f^{(r)}$ degenerates at most over $D_f$, it follows 
that there exists an inclusion
$$L_m^{\otimes N} (- kD_f) \hookrightarrow f^{(r)}_* \omega_{Y^{(r)}/X}^{\otimes m}$$
for some integer $k \ge 0$. This implies in particular that on $Y^{(r)}$ we have 
$$H^0 \big(Y^{(r)}, \omega_{Y^{(r)}/X}^{\otimes m} \otimes {f^{(r)}}^* M^{-1} (kD_f) \big) \neq 0.$$
Thus we can take $(Y^{(r)} , f^{(r)})$ to play the role of $(Y', f')$ in ($\ref{wish}$).

\medskip
\noindent
{\bf Step 2.}
Fix now an ample line bundle $A$ on $X$. The considerations above show that, in order
to prove the theorem, we can assume that there exists an integer 
$k_0 \ge 0$ such that the condition in \eqref{eq:sections} is satisfied for $f'$ with respect to the line bundle 
$$L = A (-k_0D_f).$$
But then \theoremref{thm:VZ} provides a graded $\Sym \shT_X$-module $\shGb$ as in its statement, which 
in particular is large with respect to $D_f$.
\end{proof}

\subsection{General case}\label{scn:general}

This section contains the proof of \theoremref{thm:VZ-sheaves}, and explains how to deduce \theoremref{thm:hyperbolicity} in 
the general case. 
The main reason \theoremref{thm:VZ-sheaves} is better suited for the argument is that the integer $r$ in the 
statement of \theoremref{thm:submodule-positivity} could be very large, precluding its use in the uniruled case in a similar way to the previous section.

\begin{proof}[Proof of \theoremref{thm:VZ-sheaves}]
{\bf Step 1.}
We refine Step $1$ in the proof of \theoremref{thm:submodule-existence} to obtain a stronger statement under the assumption that the bigness of the determinant of some pluricanonical image holds on any generically finite cover. We claim that for any line bundle $A$ on $X$ one can modify the picture to a new family $f': Y' \rightarrow X$ such that:
$$B = \omega_{Y'/X} \otimes {f'}^*A^{-1} \,\,\,\,{\rm satisfies} \,\, (\ref{eq:sections}).$$

This requires using an extra semistable reduction in codimension one procedure, following work of Viehweg. Indeed, for instance \cite[Lemma 6.1]{Viehweg1} says that there is a commutative diagram 
$$
\begin{tikzcd}
Y  \dar{f} 
		& \tilde Y \lar \dar{\tilde f} \\
X & \tilde X \lar{\tau}
\end{tikzcd}
$$
with $\tau$ generically finite, $\tilde X$ and $\tilde Y$ smooth and projective, and 
after removing a closed subset $Z$ of codimension at least $2$ in $X$, $\tau$ is finite and flat and $f^\prime$ is semistable.\footnote{Although not strictly necessary here, note that an even stronger construction based on the Abramovich-Karu weak semistable reduction can be considered; see e.g. \cite[Lemma 2.6]{VZ2}.} Since our goal is to check the inclusion of line bundles into various torsion-free sheaves, we can ignore $Z$ and assume that these properties hold for the full diagram above. (However
$\tilde f$ may have a bigger degeneracy locus than $\tau^{-1} (D_f)$.)

For any given $r$, we again consider a resolution of singularities $Y^{(r)}$ of the main component of the $r$-fold fiber
product, with its induced morphism $f^{(r)}: Y^{(r)} \rightarrow X$, and similarly for 
$\tilde f: \tilde Y \rightarrow \tilde X$,  chosen in such a way that we  have a commutative diagram 
$$
\begin{tikzcd}
Y^{(r)}  \dar{f^{(r)}} 
		& {\tilde Y}^{(r)} \lar \dar{{\tilde f}^{(r)}} \\
X & \tilde X \lar{\tau}.
\end{tikzcd}
$$

Pick now a  line bundle on $X$ of the form $L = A^{\otimes m} \otimes M$, where $M$ is an ample line 
bundle chosen such that the sheaf $M \otimes (\tau_* \shO_{\tilde X})^\vee$ is globally generated.
By hypothesis there exists an integer $m > 0 $ such that 
$\det {\tilde f}_* \omega_{\tilde Y/ \tilde X}^{\otimes m}$ is big. Therefore exactly as in 
Step~1 in the proof of \theoremref{thm:submodule-existence}, for some sufficiently large $r$ we obtain an inclusion
$$\tau^*L \hookrightarrow \big({\tilde f}_* \omega_{\tilde Y/ \tilde X}^{\otimes m}\big)^{\otimes r}.$$
But since $\tilde f$ is semistable, by \cite[Lemma 3.2]{Viehweg1} (see also \cite[Corollary 4.11]{Mori}) the natural morphism 
$${\tilde f}^{(r)}_* \omega_{{\tilde Y}^{(r)}/\tilde X}^{\otimes m} \longrightarrow \big( 
({\tilde f}_* \omega_{\tilde Y/ \tilde X}^{\otimes m})^{\otimes r}\big)^{\vee \vee}$$
is in fact an isomorphism.  On the other hand, since $\tau$ is flat, by \cite[Lemma 3.2]{Viehweg1} (see also \cite[4.10]{Mori}) we have an inclusion
$${\tilde f}^{(r)}_* \omega_{{\tilde Y}^{(r)}/ \tilde X}^{\otimes m} \hookrightarrow \tau^* f^{(r)}_* \omega_{Y^{(r)}/ X}^{\otimes m}.$$
Putting everything together, we obtain an inclusion
$$\tau^*L \hookrightarrow \tau^* f^{(r)}_* \omega_{Y^{(r)}/ X}^{\otimes m},$$
and consequently an inclusion
$$L \hookrightarrow  f^{(r)}_* \omega_{Y^{(r)}/ X}^{\otimes m} \otimes \tau_* \shO_{\tilde X}.$$
Because of the global generation of $M \otimes (\tau_* \shO_{\tilde X})^\vee$, this in turn induces a sequence of inclusions
$$A^{\otimes m} \hookrightarrow  f^{(r)}_* \omega_{Y^{(r)}/ X}^{\otimes m} \otimes \tau_* \shO_{\tilde X} \otimes M^{-1}
\hookrightarrow \bigoplus f^{(r)}_* \omega_{Y^{(r)}/ X}^{\otimes m},$$
Finally, this provides a non-trivial homomorphism 
$A^{\otimes m} \rightarrow f^{(r)}_* \omega_{Y^{(r)}/ X}^{\otimes m}$, which means that we can take 
$(Y^{(r)} , f^{(r)})$ to play the role of $(Y', f')$ as desired.

\medskip
\noindent
{\bf Step 2.} 
Fix now a line bundle $A$ on $X$ such that $A (-D_f)$ is ample.
Using Step $1$, and switching back to the notation $f: Y \rightarrow X$, we can now assume that condition ($\ref{eq:sections}$) is satisfied for $f$ with respect to $L = A$, so that \theoremref{thm:VZ} applies to this set-up. By passing to a birational model of the base, and therefore assuming 
only that $A (-D_f)$ is big and nef,  we can arrange in addition that the singularities of the Hodge module $M$ constructed in \theoremref{thm:VZ} occur along a simple normal crossings divisor $D$ which contains $D_f$.

This is the context where \theoremref{thm:Higgs} applies; we use it to obtain a graded 
submodule $\shFb \subseteq \shEb$ of a (graded logarithmic) Higgs bundle 
\[
	\theta \colon \shEb \to \OmX^1(\log D) \tensor \shE_{\bullet+1}, 
\]
such that
\[\theta(\shFb) \subseteq \OmX^1(\log D_f) \tensor \shF_{\bullet+1}
\]
and moreover
$$A(-D_f) \subseteq \shF_0 \,\,\,\,\,\, {\rm and} \,\,\,\,\,\, \shF_k = 0 ~~{\rm for } ~ ~k < 0.$$
We can then apply \theoremref{thm:submodule-Higgs} to obtain the desired conclusion.
\end{proof}

\begin{proof}[Proof of \theoremref{thm:hyperbolicity}]
If $f$ has maximal variation, then clearly so does any $\tilde f: \tilde Y \rightarrow \tilde X$ obtained from $f$ by
a base change $\tau: \tilde X \rightarrow X$ followed by a desingularization of $Y\times_X \tilde X$. Under our 
hypotheses it follows that \theoremref{thm:VZ-sheaves} applies, which gives the conclusion
in combination with \theoremref{thm:CP}.
\end{proof}

\subsection{On the Kebekus-Kov\'acs conjecture}\label{scn:KK}
Kebekus and Kov\'acs have proposed a natural conjecture generalizing Viehweg's hyperbolicity conjecture to families that are not necessarily of maximal variation. 

\begin{conjecture}[{\cite[Conjecture 1.6]{KK1}}]\label{generalized_conjecture}
If $X^{\circ}$ is smooth and quasi-projective, and $f^{\circ}: Y^{\circ} \rightarrow X^{\circ}$ is a smooth family of canonically
polarized varieties with maximal variation,  then either $\kappa (X^{\circ}) = - \infty$ and $\dim X^\circ > {\rm Var}(f)$, or 
$\kappa (X^\circ) \ge {\rm Var} (f)$. 
\end{conjecture}

In \cite{KK1} and \cite{KK3} they showed that this conjecture holds when $X^\circ$ has dimension two and three, respectively, and provided a beautiful structural analysis according to the different possible values of the variation in these cases.
In \cite{KK2} they showed that it holds when $X^\circ$  is projective of arbitrary dimension if the conjectures of the minimal model program,  including abundance, are assumed for all varieties of dimension at most $\dim X^\circ$. The conjecture is now known to 
hold in general due to work of Taji \cite{Taji}, who proved Campana's isotriviality conjecture, which in turn implies the Kebekus-Kov\'acs conjecture.

We note for completeness that, again when $ X = X^\circ$ is projective, a small variation of the methods in this paper leads to a proof of \conjectureref{generalized_conjecture} for the more general types of families of varieties considered here as well, provided that a statement along the lines of the abundance conjecture were known for $K_X$. Such a statement was conjectured by Campana and Peternell; the case $A = 0$ is a famous special case of the abundance conjecture. 

\begin{conjecture}[{\cite[Conjecture p.2]{CP}}]\label{CP_conjecture}
If $X$ is a smooth projective variety and  
$$K_X \sim_{\QQ} A + B,$$
where $A$ and $B$ are an effective and a pseudo-effective $\QQ$-divisor respectively, then  $\kappa (X) \ge \kappa (A)$. 
\end{conjecture}

Consider now a smooth family $f \colon Y \to X$, with $X$ and $Y$ smooth and projective. We assume that the fibers are of general type, or more generally that the geometric generic fiber has a good minimal model. 

\medskip

\noindent
\emph{Claim:} ~\conjectureref{generalized_conjecture} holds for $f$ assuming that \conjectureref{CP_conjecture} holds for $X$.

\medskip

To see this, we can assume that $\kappa (X) \ge 0$, 
since in the uniruled case \theoremref{thm:hyperbolicity} already implies that ${\rm Var}(f) < \dim X$. 
In any case, for all families with  fibers as assumed,  Kawamata \cite[Theorem 1.1]{Kawamata} has shown a stronger version of the $Q_{n,m}$ conjecture: for $m\ge 1$ sufficiently large one has 
$$\kappa ( \det f_* \omega_{Y/X}^{\otimes m}) \ge {\rm Var} (f).$$
This implies via an argument similar to that in \theoremref{thm:submodule-existence} (and in fact simpler, since now
$D = \emptyset$) that there exists a graded submodule  
$\shF_{\bullet}$ of a graded Higgs bundle $\shE_{\bullet} \to \Omega_X^1 \otimes \shE_{\bullet + 1}$, 
and a line bundle with $\kappa (A) \ge {\rm Var}(f)$, such that 
$$A \subseteq \shF_0 \,\,\,\,\,\, {\rm and} \,\,\,\,\,\, \shF_k = 0 ~~{\rm for } ~ ~k < 0.$$
This in turn produces for some $s \ge 1$ a subsheaf $\mathcal{H} \subseteq (\Omega_X^1)^{\otimes s}$ as in 
\theoremref{thm:submodule-Higgs}, only this time $\mathcal{H}$ is not big, but rather only has the property that $\mathcal{H} \simeq \mathcal{\shG} \otimes A$,
where $\shG$ is a weakly positive sheaf. Indeed, repeating the argument used in the proof of  \theoremref{thm:submodule-positivity} and \theoremref{thm:submodule-Higgs}, we obtain a morphism
$$K_{p+s}^\vee \otimes A \longrightarrow (\Omega_X^1)^{\otimes s},$$
where $K_{p+s}^\vee$ is weakly positive, and again we take $\shH$ to be its image. In particular we have 
$\det \mathcal{H} \simeq \det \shG \otimes B$ with $\det \shG$ pseudo-effective, and $\kappa (B) \ge {\rm Var}(f)$.
Moreover, the inclusion of $\shH$ induces an exact sequence
$$0 \longrightarrow \mathcal{H} \longrightarrow (\Omega_X^1)^{\otimes s} \longrightarrow \shQ  \longrightarrow 0.$$
As at the end of the proof of \theoremref{thm:submodule-positivity}, we can assume that $\shH$ is saturated, and therefore
by the same result of Campana-P\u aun that $\det \shQ$ is pseudoeffective. Passing to determinants we obtain
$$\omega_X^{\otimes s} \simeq \det \shQ \otimes \det \shG \otimes B,$$
where the first two line bundles on the right are pseudo-effective. At this stage \conjectureref{CP_conjecture} implies that 
$$\kappa (X) \ge \kappa (B) \ge {\rm Var}(f),$$
as predicted by \conjectureref{generalized_conjecture}.

\bibliographystyle{amsalpha}
\bibliography{bibliography}

\end{document}